%% file: compact-singular-lopez-finalversion.tex
\documentclass[11pt]{amsart}

\usepackage{amsfonts}
\usepackage{amssymb}
\usepackage{amsthm}
\usepackage[centertags]{amsmath}
\usepackage{newlfont}
\usepackage{graphicx}
\usepackage{color}

\def\r{\mathbb R}
\def\s{\mathbb S}

\newtheorem{theorem}{Theorem}[section]

\newtheorem{proposition}[theorem]{Proposition}
\newtheorem{remark}[theorem]{Remark}
\newtheorem{corollary}[theorem]{Corollary}

\begin{document} 

\title{Compact singular minimal  surfaces with boundary}
\author{Rafael L\'opez}
 \thanks{Partially supported by the grant no. MTM2017-89677-P, MINECO/AEI/FEDER, UE}
 \address{Departamento de Geometr\'{\i}a y Topolog\'{\i}a\\ Instituto de Matem\'aticas (IEMath-GR)\\
 Universidad de Granada\\
 18071 Granada, Spain}
  \email{ rcamino@ugr.es}
\date{}

\begin{abstract} We study  of the shape of a compact singular minimal surface in terms of the geometry of its boundary, asking what type of {\it a priori} information  can be obtained on  the surface from the knowledge of its boundary.  We derive  estimates of the area and the height in terms of the boundary. In case that  the boundary is a circle, we study under what conditions the surface is rotational. Finally, we deduce non-existence results when    the boundary is formed by two curves that are sufficiently far apart.\end{abstract}


\maketitle
\section{Introduction}

In this paper we consider the following variational problem in three-dimensional Euclidean space $(\r^3,\langle,\rangle)$. Let $\alpha\in\r$ denote some real number and let   $\vec{a}\in\r^3$ be a unit vector. Consider the halfspace  $\r_+^3(\vec{a})=\{p\in\r^3:\langle p,\vec{a}\rangle>0\}$ and the plane $\r^3_0(\vec{a})=\{p\in\r^3:\langle p,\vec{a}\rangle=0\}$. Let  $\phi:M\rightarrow\r^3_+(\vec{a})$ be a smooth immersion of an oriented compact surface $M$ with non-empty boundary $\partial M$.  We define the potential energy  of $\phi$ in the direction of $\vec{a}$ by
$$E(\phi)=\int_M \langle \phi(p),\vec{a}\rangle^\alpha\ dM,$$
where $dM$ is the measure on $M$ with respect to the induced metric. Let $\Phi:M\times(-\epsilon,\epsilon)\rightarrow\r_+^3(\vec{a})$ be  a compactly supported  variation of $\phi$ fixing   $\partial M$  whose variational vector field is   $\xi=(\partial\Phi/\partial t)_{t=0}$.  By letting $E(t)=E(\Phi(-,t))$,    the first variation of $E$ is
$$E'(0)=-\int_M \left(2H-\alpha\frac{\langle N,\vec{a}\rangle}{ \langle\phi,\vec{a}\rangle}\right)\langle\phi,\vec{a}\rangle^\alpha u\ dM,$$
where $N$ and $H$ denote the Gauss map and the mean curvature of  $\phi$, respectively, and $u=\langle N,\xi\rangle$: see \cite{di1,di2}.  Then $\phi$ is a critical point of $E$ for all compactly supported variations of $\phi$  if and only if
\begin{equation}\label{mean}
2H=\alpha \frac{\langle N,\vec{a}\rangle}{\langle \phi,\vec{a}\rangle}\quad \mbox{on $M$.}
\end{equation}
Following Dierkes in \cite{di2}, an immersed surface $M$ in $\r^3_{+}(\vec{a})$ that satisfies (\ref{mean}) is called a  {\it singular minimal surface}: in case that we want to emphasize the constant  $\alpha$, we   say that $M$ is  an {\it $\alpha$-singular minimal surface}. We refer \cite{di2,ni} for a historical introduction and a non-exhaustive list of references includes:  \cite{bht,di,dh,dt,ke,lo,lo3,te,wi}

The parenthesis in the expression of $E'(0)$   also appears in the context of manifolds with density if we see the term $\langle\phi,\vec{a}\rangle^\alpha dM $ as a way to weight the surface area.  In general, let $e^\varphi$ be a   positive density function in $\r^3$, $\varphi\in C^\infty(\r^3)$, which serves as a weight for  the volume and the surface area (\cite{gr,mo}). Note that  this is not equivalent to scaling the metric conformally by  $e^\varphi$ because the area and the volume   change with different scaling factors. Denote by $dA_\varphi=e^\varphi dM $ the area element of $M$ with  density $e^\varphi$. Consider a variation $\Phi$ of $M$ as above and let    $A_\varphi(t)$ and $V_\varphi(t)$ be  the weighted area and the enclosed weighted  volume  of $\Phi(M\times\{t\})$, respectively. Then the first variation of $A_\varphi(t)$ and $V_\varphi(t)$ are, respectively, 
$$A'_\varphi(0)=-2\int_M H_\varphi u\  dA_\varphi,\quad V_\varphi'(0)=\int_M u\ dA_\varphi,$$
where $H_\varphi=H-\langle\nabla\phi,N\rangle/2$ is called the {\it weighted mean curvature} (\cite{cab}).  If we choose $\varphi(q)=\log(\langle q,\vec{a}\rangle^\alpha)$, then 
$$H_\varphi=H-\alpha\frac{\langle N,\vec{a}\rangle}{2\langle\phi,\vec{a}\rangle}.$$
As a consequence of the Lagrange multipliers,     $M$ is a critical point of the  area $A_\varphi$ for a  prescribed weighted volume if and only if  $H_\varphi$ is a constant function. Equation  (\ref{mean}) corresponds simply with a surface that is a critical point of $A_\varphi$, namely,  $H_\varphi=0$. We vector  $\vec{a}$ is called the {\it density vector}.

Some special cases of $\alpha$-singular minimal surfaces are: if $\alpha=0$, then $M$ is a minimal surface;  if $\alpha=-2$, $M$ is a minimal surface in   $3$-dimensional hyperbolic space  when this space is viewed in the upper half-space model. A remarkable case is   $\alpha=1$, because these surfaces have the minimal potential energy under gravitational forces, or in other words, have the lowest center of gravity. These surfaces  were coined in the literature as the ``two-dimensional analogue of the catenary'' because   the catenary is the planar curve with the lowest center of gravity  (\cite{bht,dh}).   

In this paper, we will assume $\alpha\not=0$.  We are interested in those singular minimal surfaces with  special geometric properties. A first family of surfaces  are the  cylindrical surfaces, that is,  ruled surfaces where all the rulings are parallel.   Although there is not an {\it a priori } relation between the direction of the rulings and the density vector $\vec{a}$, it is not difficult to see that either the surface is a  parallel plane    to  $\vec{a}$ or the   rulings must be orthogonal to $\vec{a}$  (\cite{lo}). In the latter case, the surface is generated by a planar curve $\gamma:I\rightarrow\r^2$ whose  curvature $\kappa$     satisfies the one-dimensional case of (\ref{mean}) 
   \begin{equation}\label{cate}
 \kappa(s)=\alpha\frac{\langle{\textbf n}(s),\vec{w}\rangle}{\langle\gamma(s),\vec{w}\rangle},\qquad(s\in I),
 \end{equation}
 where $\vec{w}\in\r^2$ is a unit vector and ${\textbf n}$ is the unit principal normal vector of $\gamma$. In such a case, we say that $\gamma$ is an {\it $\alpha$-catenary} and the corresponding ruled surface is called an {\it $\alpha$-catenary cylinder}. See Figure  \ref{figex}, left. When $\alpha=1$, $\gamma$ is, indeed,  a catenary.  

 The second family of singular minimal surfaces are those ones of rotational type. It was proved in \cite{lo} that either the rotation axis  is parallel to the density vector $\vec{a}$ or the rotation axis  is contained in the plane $\r^3_{0}(\vec{a})$. In the first case, and for $\alpha>0$, these  surfaces were studied by Keiper in  an unpublished paper \cite{ke} but  well known for specialists (\cite{bd,di,di1,di2}). For this class of surfaces, there are   two types of rotational surfaces depending if  the surface meets (necessarily orthogonally) the rotational axis (Figure  \ref{figex}, right) or if it does not meet it and in such a case,  the surface is winglike-shaped  (Figure  \ref{figex2}, left).   A limit  case occurs when the profile curve meets the plane $\r^3_0(\vec{a})$, being  the surface  a   half-cone  (Figure  \ref{figex2}, right). In case that  the rotation axis lies in $\r^3_0(\vec{a})$,  the profile curve is just an $(\alpha+1)$-catenary (\cite{lo}): see Figure  \ref{figex3}.

\begin{figure}[hbtp]
\begin{center}
\includegraphics[width=.4\textwidth]{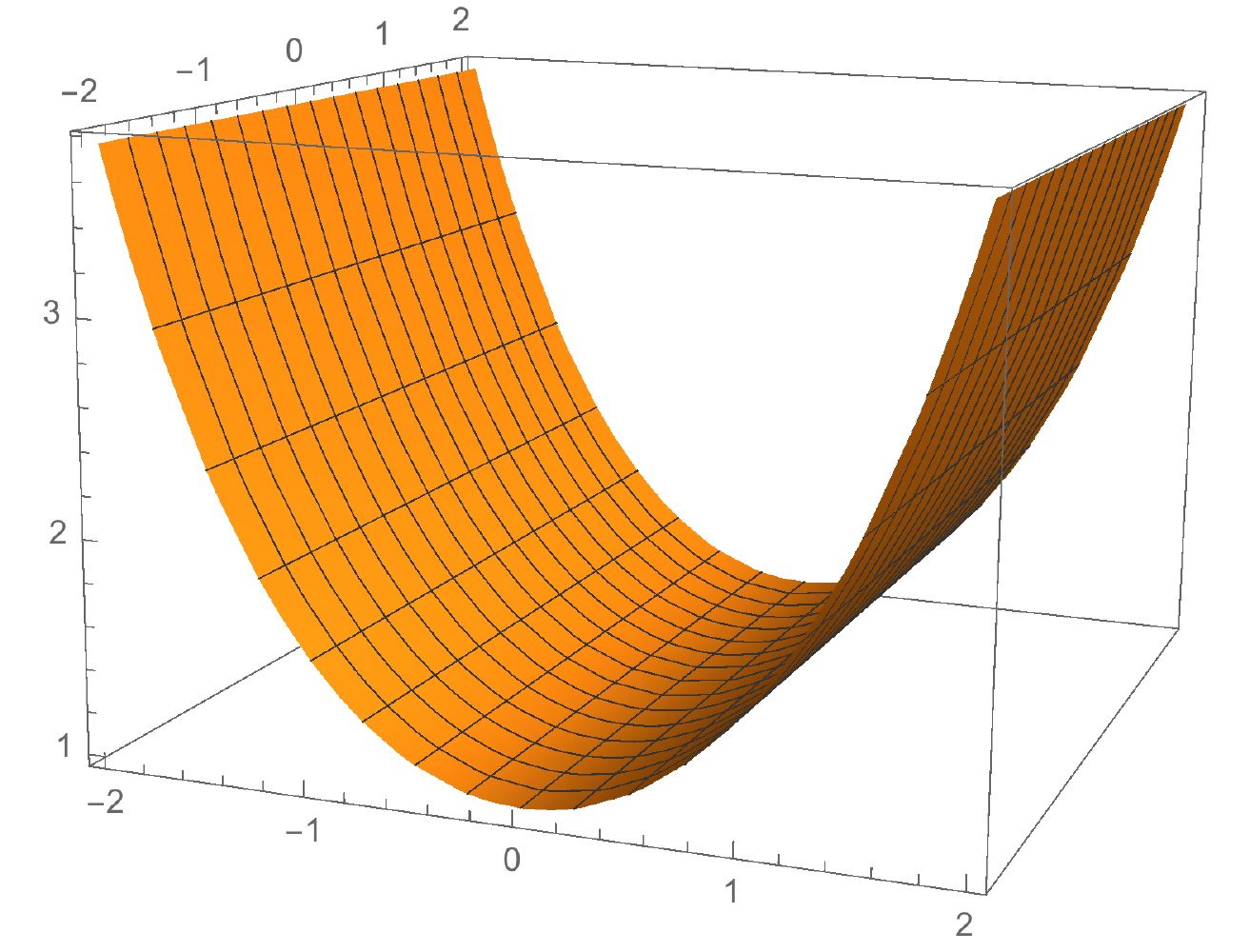}  \includegraphics[width=.4\textwidth]{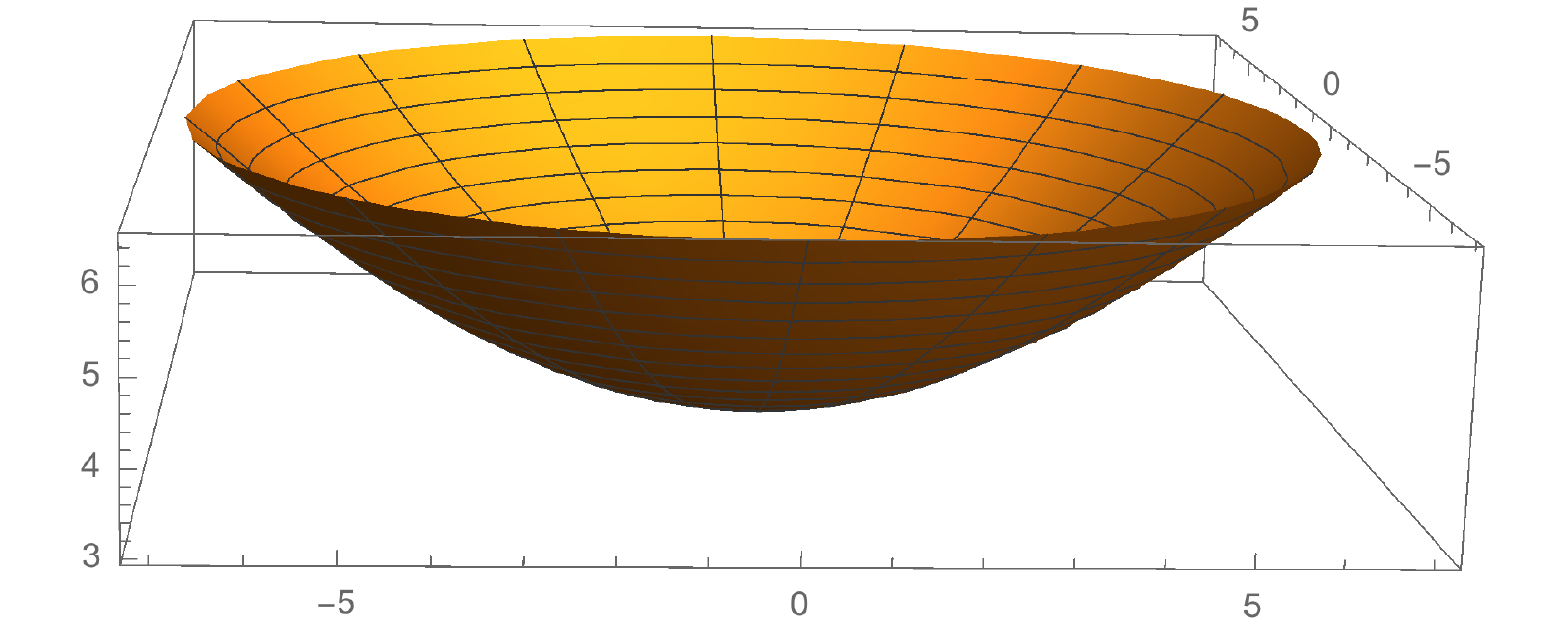}
\end{center}
\caption{The catenary cylinder (left) and a rotational  $\alpha$-singular minimal surface (right) for $\alpha=1$}\label{figex}
\end{figure}

\begin{figure}[hbtp]
\begin{center}
\includegraphics[width=.4\textwidth]{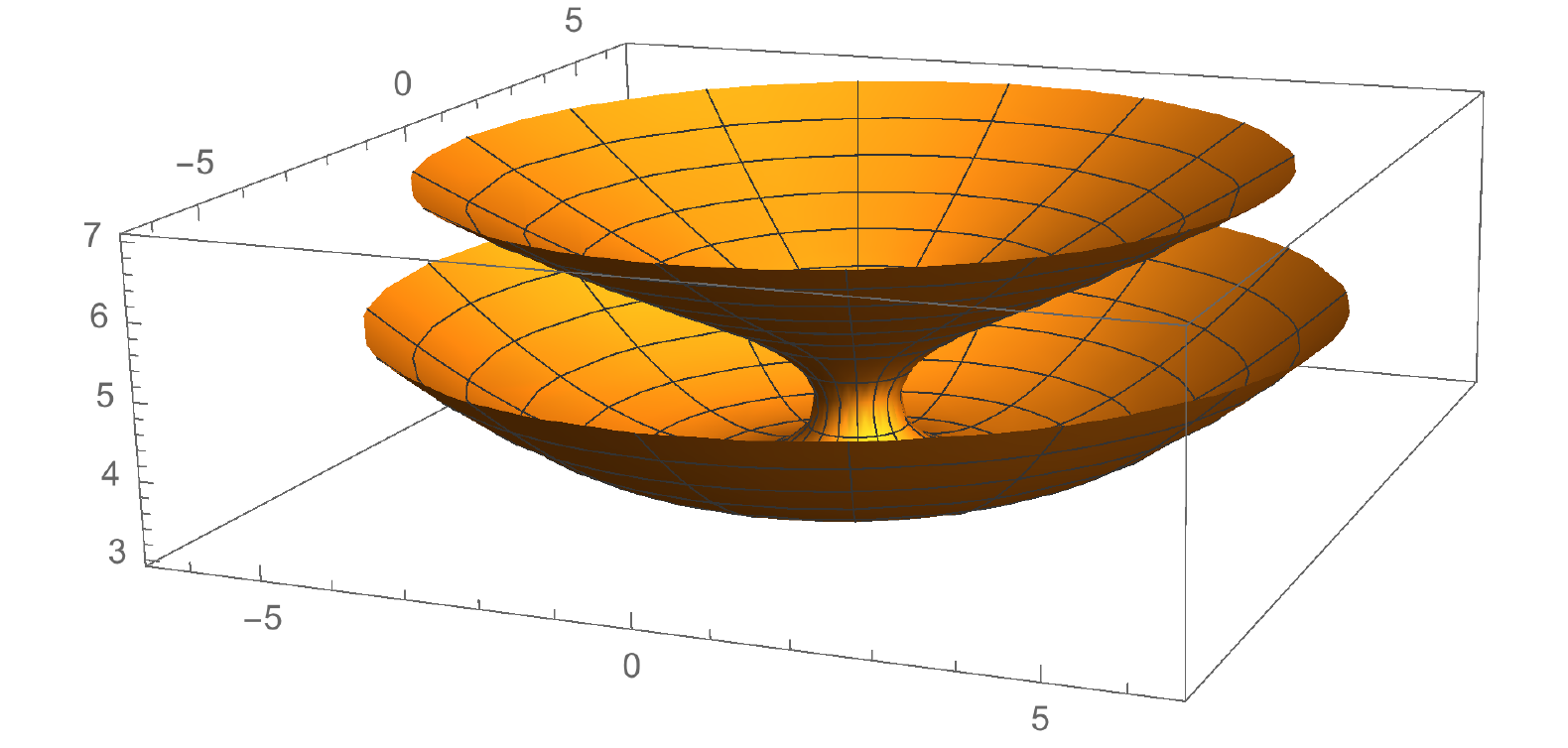}  \includegraphics[width=.4\textwidth]{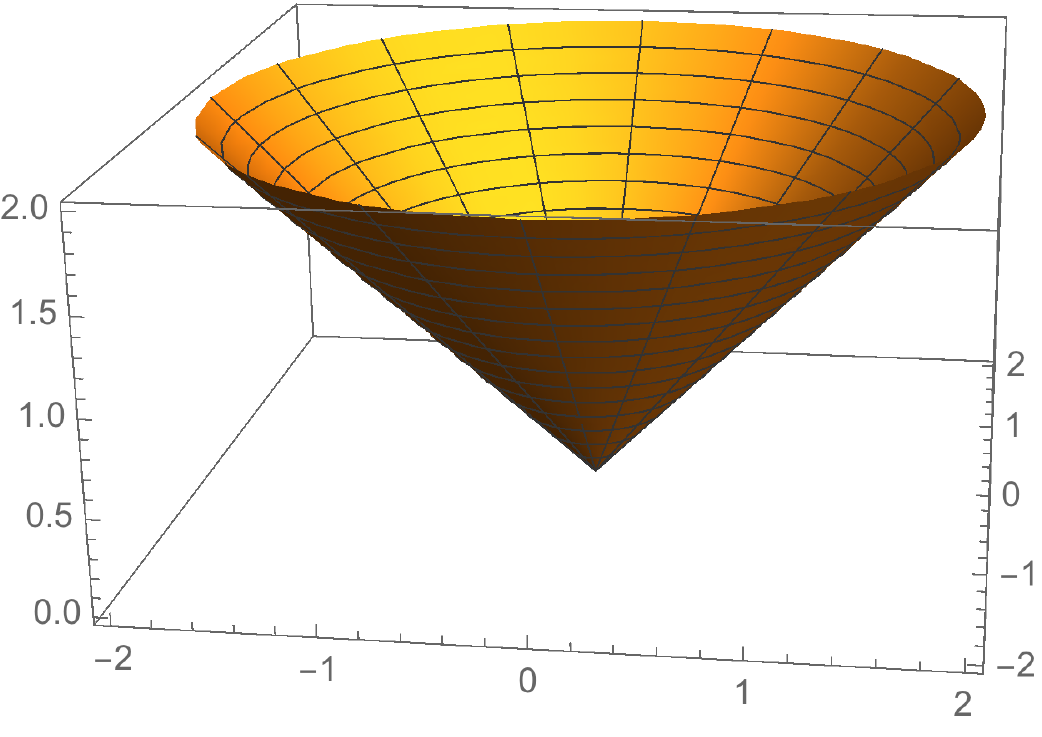}
\end{center}
\caption{A winglike $\alpha$-singular minimal surface (left)  and a half-cone (right) for $\alpha=1$}\label{figex2}
\end{figure}
\begin{figure}[hbtp]
\begin{center}
\includegraphics[width=.3\textwidth]{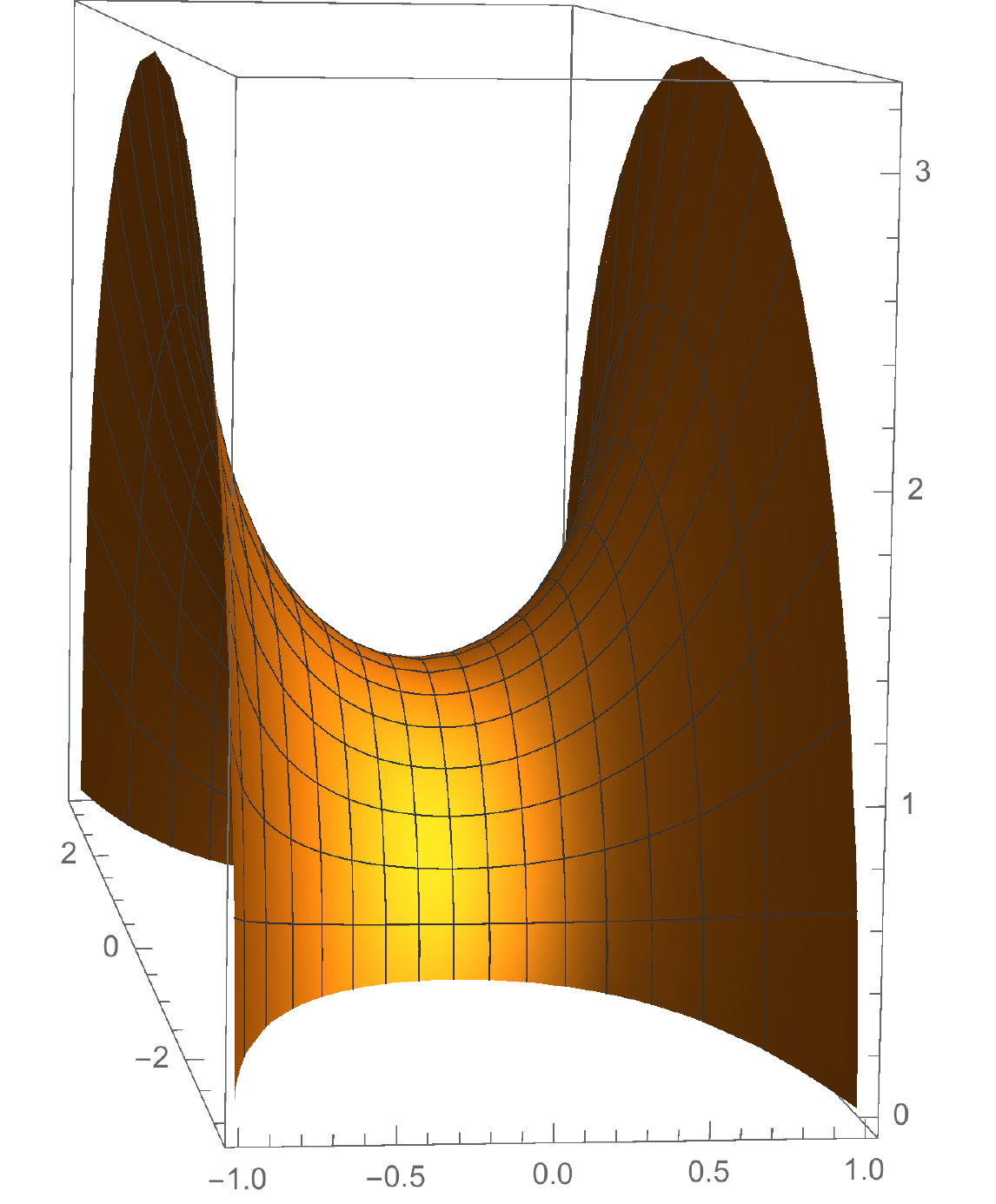} \end{center}
\caption{A $1$-minimal surface that is rotational about the $x$-axis. Here $\vec{a}=(0,0,1)$ and the generating curve is a $2$-catenary}\label{figex3}
\end{figure}
We precise the terminology. Let $\Gamma\subset\r^3$ be a closed curve and let $M$ be a   surface with smooth boundary $\partial M$. As usual, $M$ is said to be a surface with boundary $\Gamma$ if there exists an immersion $\phi:M\rightarrow\r^3_{+}(\vec{a})$ such that the restriction of  $\phi$   to   $\partial M$ is a diffeomorphism onto $\Gamma$. When $\phi$ is an embedding,   we write simply $\partial M$ instead of $\Gamma$.

The purpose of this paper is the  study  of the shape of a compact singular minimal surface in terms of the geometry of its boundary, asking what type of {\it a priori} information  can be obtained on  the surface from the knowledge of its boundary. Notice   that there do not exist closed surfaces with constant weighted mean curvature (Proposition \ref{p-not}) and thus, a compact singular minimal surface has non-empty boundary. Some  enclosure results were established in \cite{bht} by means of  the maximum principle. So,  if $\partial M$ is contained in a certain domains  of a solid cone (or a paraboloid), then $M$ is included in the same domain.  

A first question that we address is:

\begin{quote} {\bf Problem A.} {\it Find estimates of the size of a compact singular minimal surface  in terms of  its boundary.}
\end{quote}

In Section  \ref{s-area} we deduce   a priori estimates of the   surface area and in Section   \ref{s-height} we derive  height estimates   for a rotational singular minimal surface in terms of the height of its boundary.

Other interesting question concerns to the following  

\begin{quote} {\bf Problem B.} {\it Under what  conditions a compact singular minimal surface  inherits the symmetries of its boundary.}
\end{quote}

For example,  we ask if a singular minimal surface with boundary a circle   is a surface of revolution.  In Section     \ref{s-b1} we prove affirmatively   in the special case that the circle is  contained in a plane orthogonal to the  vector $\vec{a}$,  proving in Corollary   \ref{c-c1}:

\begin{quote}{\it  Surfaces of revolution are the only  embedded compact singular minimal surfaces with boundary a circle contained in an orthogonal plane to the density vector $\vec{a}$.}
\end{quote}

If $\alpha<0$, we can drop the embeddedness assumption in te above result (Proposition \ref{pr-ne}). Similarly, in Corollary   \ref{c-c2} we prove that if an  embedded compact singular minimal surface with boundary contained in a plane orthogonal to $\vec{a}$  makes a contact angle with the boundary plane, then the surface must be rotational and its boundary is a circle.  The main ingredient  in both results is the Alexandrov reflection method, which uses  the maximum principle for the singular minimal surface equation (\ref{mean}).

In Section     \ref{s-b2} we give non-existence results of singular minimal surfaces spanning two given curves. Roughly speaking, we prove  that if   two curves are far sufficiently apart, then they can not span a   compact connected singular minimal surface: the distance between both curves also depends on the separation with respect to the plane $\r^3_{0}(\vec{a})$.

{\bf Acknowledgements.} Part of this paper was done by the author in 2016 during a  visit in  the Department of Mathematics of the   RWTH Aachen University. The author thanks specially to Prof. Josef Bemelmans for his valuable discussions and hospitality.

 \section{Preliminaries and the tangency principle}\label{s-pre}

We stand for $(x,y,z)$ the canonical coordinates of Euclidean space $\r^3$ and we shall use the terminology horizontal and vertical as usual, where $z$ indicates the vertical direction: in general, a horizontal (resp. vertical) direction means orthogonal (resp. parallel) to the density vector.

It is   immediate that the    singular minimal surface equation  (\ref{mean}) is invariant by some transformations of Euclidean space. For example,  if $M$ is an $\alpha$-singular minimal surface and   $T$ is a translation along a horizontal direction   or $T$ is a rotation about a vertical line, then $T(M)$ is also an $\alpha$-singular minimal surface. The same occurs if $T$ is a dilation  with positive ratio and center any point of $\r^3_0(\vec{a})$. We  point out that if we reverse the orientation on $M$, $H$ changes of sign, hence the equation (\ref{mean}) is preserved.

 In order to study the   behavior of a singular minimal surface,  we look equation (\ref{mean}) in  local coordinates.  After a change of coordinates, we suppose that the density vector is $\vec{a}=(0,0,1)$. Then it is immediate that  equation (\ref{mean}) in  nonparametric form    is 
\begin{equation}\label{mean1}
\mbox{div}\left(\frac{Du}{\sqrt{1+|Du|^2}}\right)=\frac{\alpha}{u\sqrt{1+|Du|^2}},
\end{equation}
where $z=u(x,y)$. Equation (\ref{mean1})  is of elliptic type and  satisfies a maximum principle that   can be formulated as follows (\cite[Th. 10.1]{gt}). 

\begin{proposition}[Tangency principle] Suppose $M_1$ and $M_2$ two surfaces with weighted mean curvatures $H_\varphi^1$ and $H_\varphi^2$ respectively. Suppose $M_1$ and $M_2$ are tangent at a common interior point $p$ and the corresponding Gauss maps $N_1$ and $N_2$ coincide at $p$. If $H_\varphi^2\leq H_\varphi^1$ in a neighborhood of $p$, then it is not true that $M_2$ lies above $M_1$ near $p$ with respect to $N_1(p)$, unless $M_1=M_2$ in a neighborhood of $p$. If $p\in\partial M_1\cap \partial M_2$ is a boundary point, the same holds if, in addition,  we have $T_p\partial M_1=T_p\partial M_2$. 
\end{proposition}

If $M_1$ and $M_2$ are two singular minimal surfaces, the above  condition on the orientations at the common point can be dropped because $H_\varphi=0$ holds for any orientation.  Other consequence  of (\ref{mean1}) is that a  surface with constant weighted mean curvature   is real analytic and consequently, if two     surfaces with $H_\varphi=c$ coincide in an open set, they coincide everywhere. Given a singular minimal surface, the tangency principle allows to get information of its geometry by comparing with other singular minimal surfaces. Following this idea, we obtain the following result.

\begin{proposition}\label{prv} If $M$ is a compact singular minimal surface, then either  $M$ is a subset of a  plane parallel to the density vector $\vec{a}$ or  the function $\langle p,\vec{v}\rangle$   does not attain a global extrema at an interior point of $M$ for any  direction $\vec{v}$ orthogonal to $\vec{a}$. 
\end{proposition}

\begin{proof}
 Let $p\in M$ an interior point where the distance function $\langle p,\vec{v}\rangle$ attains a global maximum (or minimum). Since $p$ is an interior point,  the affine tangent plane $T_pM$ is parallel to $\vec{a}$, hence $T_pM$ is a singular minimal surface. Since  $M$ lies in one side of $T_pM$ around the common point $p$, the tangency principle and the analyticity   implies that $M$ is a subset of   $T_pM$. \end{proof}

As a  consequence, if the boundary is contained in a plane, we  conclude (see also \cite[Th. 4]{bht}):

\begin{corollary}\label{prv2} Let  $M$ be a compact singular minimal  surface whose  boundary $\Gamma$ is contained in a plane $\Pi$.
\begin{enumerate}
\item If  $\Pi$ is parallel to   $\vec{a}$, then $M$ is contained in  $\Pi$.
\item If $\Pi$ is not parallel to $\vec{a}$, then $\mbox{int}(M)\subset \Omega\times\r\vec{a}$, where $\Omega\subset \r^3_0(\vec{a})$ is the bounded domain by the convex hull of  $\pi(\Gamma)$, and $\pi:\r^3\rightarrow \r^3_0(\vec{a})$ is the orthogonal projection.\end{enumerate}
\end{corollary}

 For example, if $\vec{a}=(0,0,1)$ and $\Gamma$ is a convex curve contained in a horizontal plane, we deduce that $\mbox{int}(M)\subset\Omega\times\r$, where $\Omega\subset\r^2=\r^2\times\{0\}$ is the domain bounded by $\pi(\Gamma)$.
 
Similarly as in Proposition    \ref{prv}, we can compare a singular minimal surface with   planes orthogonal to $\vec{a}$.   Suppose that $p\in M$ is an interior  point where the height function $g(p)=\langle p,\vec{a}\rangle$ has a local maximum. Then the tangent plane $T_pM$ at $p$  lies locally above $M$ around $p$. We consider on $M$ the  orientation $N$   that at $p$ satisfies  $N(p)= \vec{a}$. If we orient $T_pM$ so its unit normal vector field is $\vec{a}$, then $T_pM$ lies above $M$ around the point $p$. Consequently, the weighted mean curvature    of $T_pM$ at $p$, namely,  $-\alpha/(2\langle p,\vec{a}\rangle)$, should be bigger than $H_\varphi(p)=0$.  Consequently,  we have proved:

\begin{proposition}\label{pr-1} Let $M$ be an $\alpha$-singular minimal surface. If $\alpha>0$ (resp. $\alpha<0$), then the height function $h$ defined on $M$ does not attain a local maximum (resp. a local minimum) at an interior point. In particular, if $M$ is a compact surface, then 
$$ \max\{\langle p,\vec{a}\rangle:p\in M\}=\max\{\langle p,\vec{a}\rangle:p\in\partial M\}\qquad (\mbox{case } \alpha>0),$$
$$\min\{\langle p,\vec{a}\rangle:p\in M\}=\min\{\langle p,\vec{a}\rangle:p\in\partial M\}\qquad (\mbox{case } \alpha<0).$$
\end{proposition}

This result asserts that if the boundary is contained in  a  plane $\Pi$ orthogonal to $\vec{a}$, then $M$  lies in one side of $\Pi$. More exactly,  if $\alpha>0$ (resp. $\alpha<0$), $M$ lies below (resp. above) $\Pi$ with respect to the direction   $\vec{a}$.

 We finish this section proving that there are no closed (compact without boundary) singular minimal surfaces. More general, we consider the case that the weighted mean curvature  is constant.

\begin{proposition}\label{p-not} There do not exist   closed surfaces with constant weighted mean curvature.
\end{proposition}

\begin{proof}
The proof is by contradiction. Suppose that $M$ is a closed surface with constant weighted mean curvature $H_\varphi=c$. Consider the height function $g:M\rightarrow\r$  defined by $g(p)=\langle p,\vec{a}\rangle$. The function $g$ satisfies 
\begin{equation}\label{lapla}
\Delta g=2H\langle N,\vec{a}\rangle,
\end{equation}
 where $\Delta$ is the Laplace-Beltrami operator on $M$ with respect to the induced metric. Notice that (\ref{lapla}) holds for any immersion $\phi$. It follows from     (\ref{mean}) and (\ref{lapla})  that
\begin{equation}\label{laplaciano}
\Delta g=2c\langle N,\vec{a}\rangle+\alpha\frac{\langle N,\vec{a}\rangle^2}{\langle p,\vec{a}\rangle}.
\end{equation}
An integration on $M$  of this identity and the divergence theorem yield
$$0=c\int_M \langle N,\vec{a}\rangle\ dM +\alpha\int_M \frac{\langle N,\vec{a}\rangle^2}{\langle p,\vec{a}\rangle}\ dM$$
because   $\partial M=\emptyset$. On the other hand,  $\int_M\langle N,\vec{a}\rangle\ dM=0$   in any a closed surface because the vector field $\vec{a}$ is divergence free. Then  
$$\alpha\int_M \frac{\langle N,\vec{a}\rangle^2}{\langle p,\vec{a}\rangle}\ dM=0.$$
Because $\alpha\not=0$ and   $\langle p,\vec{a}\rangle>0$ on $M$, we deduce that $\langle N,\vec{a}\rangle=0$ on $M$. This a contradiction because for a closed surface, the Gauss map $N$ is onto in the unit sphere $\s^2$. 
\end{proof}

As a consequence of Proposition    \ref{p-not}, any compact singular minimal surface $M$ has non-empty boundary $\partial M$. We denote the {\it interior} of $M$ as $\mbox{int}(M)=M\setminus \partial M$.

\begin{remark} \begin{enumerate}
\item For singular minimal surfaces,   Proposition \ref{p-not} may be proved by using Proposition  \ref{prv}   comparing the surface with vertical planes and applying the tangency principle. 
\item Proposition \ref{p-not} contrasts to what happens in the family of surfaces of $\r^3$ with constant mean curvature, where there exist many examples of closed surfaces with arbitrary genus.
\end{enumerate}
\end{remark}

 \section{Area estimates for singular minimal surfaces}\label{s-area}

 Concerning to  the Problem A, in this section we obtain a relation between the area  of a compact  singular minimal surface and its boundary. We will prove that  the area  can not be arbitrary large, in fact, its  area is less than some constant  depending on the boundary. First, we study the case  $\alpha\geq 1$. 

\begin{proposition}\label{p-a0} 
Let $\alpha\geq 1$. Let $\Gamma\subset \r^3_+(\vec{a})$ be a closed curve and set
\begin{equation}\label{cc}
C=L(\Gamma)\sup_{p\in\Gamma}\langle p,\vec{a}\rangle,
\end{equation}
where $L(\Gamma)$ is the length of $\Gamma$. Then the area $A(M)$ of any compact $\alpha$-singular minimal surface with boundary $\Gamma$ satisfies
\begin{equation}\label{am}
A(M)\leq C.
\end{equation}
 \end{proposition}

\begin{proof}
We compute $\Delta g^2$ for the function   $g(p)=\langle p,\vec{a}\rangle$. The gradient of $g$ is $\nabla g(p)=\vec{a}^T$, where $\vec{a}^T$ is the tangent part of $\vec{a}$ on $T_p M$. Thus $|\nabla g|^2=1-\langle N,\vec{a}\rangle^2$, and from (\ref{lapla}),  
$$\Delta g^2=2g\Delta g+2|\nabla g|^2=2+2(\alpha-1)\langle N,\vec{a}\rangle^2.$$
Since  $\alpha\geq 1$,  the divergence theorem implies
\begin{equation}\label{pnu}
-\int_{\partial M} \langle p,\vec{a}\rangle\langle\nu,\vec{a}\rangle ds=\int_M\left(1+(\alpha-1)\langle N,\vec{a}\rangle^2\right)dM\geq  A(M),
\end{equation}
where $A(M)$ is the area of $M$ and $\nu$ is the unit inner conormal vector along $\partial M$.   Finally, the left hand side of the above inequality is bounded by the constant $C$ of (\ref{cc}), proving (\ref{am}).
\end{proof}

Notice  that the constant $\alpha$ does not appear in the estimate (\ref{am}). 

If  $\Gamma$ is a simple closed curve in a plane $\Pi$ orthogonal to $\vec{a}$, then $A(M)$ is bigger than the area of the planar domain bounded by $\Gamma$ in $\Pi$. Thus  we conclude:
 
 \begin{corollary}\label{c-cc} Let $\Gamma$ be a  simple closed curve contained in the plane $\Pi$ of equation $\langle p,\vec{a}\rangle=c>0$ and denote by $|\Omega|$ the area of the domain $\Omega\subset\Pi$ bounded by  $\Gamma$. If $\alpha\geq 1$, a necessary condition for the existence of a compact $\alpha$-singular minimal surface with boundary $\Gamma$ is $$|\Omega|\leq c L(\Gamma).$$
 \end{corollary}
 This inequality can be viewed as {\it a priori} condition for the existence of an $\alpha$-singular minimal surface spanning a given boundary curve $\Gamma$, because it links the position of $\Pi$,   that is, the value $c$, and the size of $\Gamma$.

 When $0<\alpha<1$,  we give estimates of the area for   graphs. 
 
 \begin{proposition}\label{pr-alpha} Let $\alpha\in (0,1)$. Let $\Gamma\subset\r^3_+(\vec{a})$ be a simple closed curve contained in the plane $\Pi$ of equation $\langle p,\vec{a}\rangle=c>0$.  If $M$ is a compact $\alpha$-singular minimal graph on $\Pi$ with boundary $\Gamma$, then 
 \begin{equation}\label{aam}
A(M)\leq cL(\Gamma)+(1-\alpha)|\Omega|,
\end{equation}
where $\Omega\subset\Pi$ is the domain bounded by $\Gamma$. As a consequence, a necessary condition for the existence of a compact $\alpha$-singular minimal graph on $\Omega\subset\Pi$ with boundary $\Gamma$ is $$ |\Omega|\leq \frac{c}{\alpha} L(\Gamma).$$
\end{proposition}
 
 \begin{proof}  By Proposition    \ref{pr-1}, we know that $M$ lies below the plane $\Pi$. Thus $M$ and $\Omega$ are two  embedded compact surfaces with the same boundary, namely, $\Gamma$. This implies that $M\cup\Omega$ bounds a $3$-domain   $\mathcal{D}\subset\r^3$ whose boundary is $M\cup\Omega$,  possibly not smooth along $\Gamma$. We choose the orientation on  $\partial \mathcal{D}$   pointing inside $\mathcal{D}$ and denote $N$ and $N_{\Omega}$ the induced orientations on $M$ and $\Omega$, respectively. Since $M$ lies below $\Pi$ and $M$ is a graph, the unit normal vector field $N$ points upwards, hence $\langle N,\vec{a}\rangle >0$ and $N_{\Omega}=-\vec{a}$. As the divergence of the vector field $\vec{a}$   is $0$, the divergence theorem in $M\cup\Omega$ gives 
 $$\int_M\langle N,\vec{a}\rangle dM=-\int_{\Omega}\langle N_{\Omega},\vec{a}\rangle=|\Omega|.$$ 
 
 On the other hand, it holds $\langle N,\vec{a}\rangle^2\leq\langle N,\vec{a}\rangle$ because $M$ is a graph on $\Pi$ and $N$ points upwards. By using $\alpha\in (0,1)$ and  (\ref{pnu}), we conclude
 $$-c\int_{\Gamma}\langle\nu,\vec{a}\rangle\geq A(M)+(\alpha-1)|\Omega|.$$
 Finally,   $-\int_\Gamma\langle\nu,\vec{a}\rangle ds\leq L(\Gamma)$, and (\ref{aam}) is proved. 
 
 \end{proof}

When $\alpha$ is negative, we derive a lower estimate of the area for singular minimal graphs.
  
 \begin{theorem}  Let $\alpha<0$. Let $\Pi$ be the plane of equation   $ \langle p,\vec{a}\rangle=c>0$. Let $M$ be a compact $\alpha$-singular minimal graph on $\Pi$ with   $\partial M\subset\Pi$. If $h$ denotes the height of $M$ with respect to $\Pi$, $h=\max_{M}\langle p,\vec{a}\rangle$, then 
\begin{equation}\label{areaestimate}
A(M)\geq -\frac{2\pi}{\alpha}(h^2-c^2).
\end{equation}
 \end{theorem}

 \begin{proof}    After a change of coordinates, we suppose that   $\vec{a}=e_3=(0,0,1)$.  Denote $\Pi(t)$ the plane of equation $z=t$, in particular,  $\Pi=\Pi(c)$. Since $\alpha<0$, we know by Proposition    \ref{pr-1} that $M$ lies above the plane $\Pi$. Let $g:M\rightarrow \r$ be the height function $g(p)=\langle p,e_3\rangle$. For each $t\geq c$, let $A(t)$ be the area of $M_t=\{p\in M: g(p)\geq t\}$ and let $\Gamma(t)$  denote the level set $\Gamma(t)=\{p\in M: g(p)=t\}$. By the coarea formula (\cite[Th. 5.8]{sa}), we have 
 $$A'(t)=-\int_{\Gamma(t)}\frac{1}{|\nabla g|} ds_t,\quad t\in\mathcal{O},$$
 where $ds_t$ is the line element of $\Gamma(t)$ and $\mathcal{O}$ is the set of all regular values of $g$. If   $L(t)$ is the length of  $\Gamma(t)$, the Schwarz inequality yields
\begin{equation}\label{lt}
L(t)^2\leq\int_{\Gamma(t)}|\nabla g|ds_t\int_{\Gamma(t)}\frac{1}{|\nabla g|}ds_t=-A'(t)\int_{\Gamma(t)}|\nabla g|ds_t.
 \end{equation}
 Along the curve $\Gamma(t)$, and since $M_t$ is above the plane $\Pi(t)$,
 $$|\nabla g|=\langle\nu^t,\nabla g\rangle=\langle\nu^t,e_3\rangle\geq 0,$$
 where $\nu^t$ is the unit inner conormal vector of $M_t$ along $\Gamma(t)$. Then (\ref{lt}) becomes  \begin{equation}\label{a1}
 L(t)^2\leq -A'(t)\int_{\Gamma(t)}\langle\nu^t,e_3\rangle ds_t.
 \end{equation}
Since $M$ is a graph, we orient $M$ with the unit normal vector field $N$  that satisfies   $\langle N,e_3\rangle> 0$ on $M$. From   (\ref{lapla}) and $\alpha<0$, it follows that 
\begin{equation}\label{a2}
\int_{\Gamma(t)}\langle\nu^t,e_3\rangle ds_t=-\alpha\int_{M_t}\frac{\langle N,e_3\rangle^2}{\langle p,e_3\rangle}dM\leq -\frac{\alpha}{t}\int_{M_t} \langle N,e_3\rangle dM,
\end{equation}
where we use that   $\langle p,e_3\rangle\geq t$ (Proposition    \ref{pr-1}) and       $\langle N,e_3\rangle^2\leq \langle N,e_3\rangle$.

We denote $\Omega(t)$ the compact surface in $\Pi(t)$ whose boundary is $\Gamma(t)$. Then $\Omega(t)$ and $M_t$ are two    embedded  compact surfaces in $\r^3$ such that their boundaries coincide. Denote by $\mathcal{D}(t)\subset\r^3$ the orientable $3$-domain that bounds $\Omega(t)\cup M_t$, possibly non-smooth  along $\Gamma(t)$. Since the orientation $N$ on $M$ points upwards,    $N$ points outside $\mathcal{D}(t)$. Let $N_{\Omega(t)}$ be the orientation on $\Omega(t)$ pointing outside $\mathcal{D}(t)$, that is, $N_{\Omega(t)}=-e_3$. If follows from the divergence theorem  
$$\int_{M_t}\langle N,e_3\rangle dM=-\int_{\Omega(t)}\langle N_{\Omega(t)},e_3\rangle=| \Omega(t)|.$$
Combining this identity with (\ref{a1}) and (\ref{a2}), we deduce
\begin{equation}\label{a3}
L(t)^2\leq \frac{\alpha}{t} A'(t)| \Omega(t)|.
\end{equation}
We decompose $\Omega(t)=\cup_{i=1}^{n_t}\Omega_i(t)$, where $\Omega_1(t),\ldots,\Omega_{n_t}(t)$ are bounded domains in $\Pi(t)$ which are determined by the closed curve $\Gamma(t)$. If $L_i(t)$ is the length of the boundary of $\Omega_i(t)$, we have 
$L(t)=\sum_{i=1}^{n_t}L_i(t)$. The   isoperimetric inequality for each domain $\Omega_i(t)$ leads to  
$$L(t)^2\geq\sum_{i=1}^{n_t}L_i(t)^2\geq 4\pi\sum_{i=1}^{n_t}| \Omega_i(t)| =4\pi| \Omega(t)|.$$
We conclude from this inequality and  (\ref{a3})
$$ A'(t)\leq  \frac{4\pi t}{\alpha},$$
because $\alpha<0$. Integrating this inequality from $t=c$ to $t=h$, we obtain the desired inequality (\ref{areaestimate}).
 \end{proof}


\section{A height estimate for rotational surfaces}\label{s-height}

 Related with the Problem A, in this section we address to the question what type of {\it a priori} information can be obtained about the height of an $\alpha$-singular minimal surface $M$ in terms of its boundary. An interesting question is to estimate how far lies $M$ from the plane $\r^3_{0}(\vec{a})$. In this section we consider compact rotational surfaces about the $z$-axis with boundary a circle and we shall give an upper bound of the height at every point of $M$ in terms of the boundary of $M$. 

After a  change of coordinates, suppose that the density vector is $\vec{a}=(0,0,1)$. 
Let $\alpha>0$. Consider $M$ a rotational $\alpha$-singular minimal surface about the $z$-axis and  intersecting  $z$-axis. Then $M$  parametrizes as $X(x,\theta)=(x\cos\theta,x\sin\theta,f(x))$ and $f:[0,\infty)\rightarrow\r^+$ satisfies:  
\begin{eqnarray}
&& \frac{f''(x)}{1+f'(x)^2}+\frac{f'(x)}{x}=\frac{\alpha}{f(x)},\label{rot1}\\
&& f(0)=z_0, f'(0)=0.\label{rot2}
\end{eqnarray}
Among the properties of $f$, we notice that the graphic of $z=f(x)$ is asymptotic to the straight-line of equation $z=\sqrt{\alpha}x$ and lies above the line $z=\sqrt{\alpha/2}x$  (\cite{ke}). Fix $r>0$ and let $S_r$ be the part of $M$  obtained by rotating $z=f(x)$ when the domain is the interval $[0,r]$.  Then $S_r$ is a compact $\alpha$-singular minimal surface with boundary    the circle $C_r=\{(r\cos\theta,r\sin\theta,f(r)):\theta\in[0,2\pi]\}$.

\begin{theorem}\label{t-esti} Let $\alpha>0$. If $f=f(x)$ is a solution of (\ref{rot1})-(\ref{rot2}), then  
\begin{equation}\label{esti}
f(x)< \left(\frac{\alpha+2}{\alpha}\right)f(r)-\sqrt{ r^2+\frac{4f(r)^2}{\alpha^2}-x^2},
\end{equation}
for all $0\leq x< r$ and $r>0$. As a consequence, the lowest point  of $f$, namely $f(0)$, satisfies
\begin{equation}\label{esti2}
f(0)<  \left(\frac{\alpha+2}{\alpha}\right)f(r)- \sqrt{  r^2+\frac{4f(r)^2}{\alpha^2}}
\end{equation}
for all $r>0$.
\end{theorem}

\begin{proof}

 Consider the  lower hemisphere centered at the $z$-axis 
 $$\s^2(R,c)=\{(x,y,z)\in\r^3:z=c-\sqrt{R^2-x^2-y^2},  x^2+y^2\leq R^2\},$$
  where $c>0$. Let us choose $\s^2(R,c)$ containing the circle $C_r$, which implies
\begin{equation}\label{r2}
R^2=r^2+(c-f(r))^2.
\end{equation}
 For  the orientation $N$ on $\s^2(R,c)$ pointing upwards, namely,  $N(p)=-(p-c)/R$,  the weighted mean curvature    $H_\varphi^S$ of  $\s^2(R,c)$ is 
\begin{equation}\label{hh}
H_\varphi^S(x,y,z)=\frac{1}{R}\left(1-\frac{\alpha(c-z)}{2z}\right).
\end{equation}
Take   $R$ and $c$ such that   $H_\varphi^S= 0$ along $C_r$: this implies  from (\ref{r2}) and (\ref{hh}) that 
\begin{equation}\label{RR}
R^2=r^2+\frac{4 f(r)^2}{\alpha^2}.
\end{equation}
Hence, and from (\ref{r2}), the value of $c$ is 
\begin{equation}\label{c3}
c=\frac{2+\alpha}{\alpha}f(r).
\end{equation}
 Consider the spherical cap  $\s_{-}^2(R,c)$ of $\s^2(R,c)$     situated below the horizontal plane of equation $z=f(r)$. From (\ref{hh}), the function $H_\varphi^S$ is increasing on $z$, and thus  $H_\varphi^S<0=H_\varphi^S(C_r)$ in the  interior of $ \s_{-}^2(R,c)$.

Denote by $\mathcal{C}$ the half-cone with vertex at the origin and determined by the circle $C_r$ and let 
$$\mathcal{C}^+=\{(x,y,z)\in\r^3: f(r)^2(x^2+y^2)<r^2z^2,  z>0\}$$
 be the $3$-domain bounded   by $\mathcal{C}$ that contains  the upper $z$-axis.

{\it Claim.} The interior of the spherical cap $\s_{-}^2(R,c)$ is included in $\mathcal{C}^+$.

In order to prove the claim, it is enough to see that 
\begin{equation}\label{ff}
\frac{f(r)}{r}< \frac{c-\sqrt{R^2-x^2}}{x}
\end{equation}
for all $x$ such that $z=c-\sqrt{R^2-x^2}$ and $0\leq x< r$. Taking into account that $f(x)>\sqrt{\alpha/2}x$, it is easily to check that  the right hand side of (\ref{ff}) is a decreasing function on $x$. Thus (\ref{ff}) is proved if 
$$\frac{f(r)}{r}\leq\frac{c-\sqrt{R^2-r^2}}{r}.$$
But this holds because of (\ref{r2}). This finishes the proof of the claim.

We are in position to prove (\ref{esti}). We observe that (\ref{esti}) is equivalent to see that the   generating curve of $\s^2_{-}(R,c)$ lies above the graphic of $z=f(x)$ for $0\leq x< r$. Indeed, $\s^2_{-}(R,c)$ is obtained by rotating the curve 
$z=c-\sqrt{R^2-x^2}$ and the inequality
$f(x)< c-\sqrt{R^2-x^2}$  coincides with (\ref{esti}) thanks to the values of $R$ and $c$ in (\ref{RR}) and (\ref{c3}), respectively. 

By using  dilations from the origin $O=(0,0,0)$, we consider the spherical caps $\lambda \s^2_{-}(R,c)$ for $\lambda$ sufficiently big so $\lambda \s^2_{-}(R,c)\cap S_r=\emptyset$. We have that all  the spherical caps $\lambda \s^2_{-}(R,c)$ are included in the domain $\mathcal{C}^+$   and that their boundaries lie contained in $\mathcal{C}$ because this occurs for $\s^2_{-}(R,c)$. Then we come back $\lambda \s^2_{-}(R,c)$ by letting $\lambda\searrow 1$. The weighted mean curvature $H_\varphi^S$ of  $\s^2_{-}(R,c)$ satisfies $H_\varphi^S< 0$. Consequently,  the value of $H_\varphi$ for all the surfaces $\lambda\s^2_{-}(R,c)$ also fulfills $H_\varphi< 0$. If there exists a first contact point $p$ between $\lambda_1 \s^2_{-}(R,c)$ and $S_r$ for some $\lambda_1>1$, then $p$ must be an interior point of both surfaces because $\partial(\lambda_1 \s^2_{-}(R,c))\subset \mathcal{C}$, $\lambda_1>1$ and $\partial \s^2_{-}(R,c)$ lies in the horizontal plane of equation   $z=f(r)$. However this is impossible by   the tangency principle    because $H_\varphi=0$ on $S_r$, $H_\varphi^S<0$ and $\lambda_1 \s^2_{-}(R,c)$ lies above $S_r$ around the point $p$.  Definitively, $\lambda=1$ and this implies that generating curve of $\s^2_{-}(R,c)$ lies above the graphic of $z=f(x)$ for $0\leq x< r$, proving the result. 
\end{proof}

\section{Singular minimal surfaces with planar boundary}\label{s-b1}

After a change of coordinates, in this section we assume that the density vector is $\vec{a}=(0,0,1)$.  Denote by $\r^3_{+}$ the halfspace $z>0$.  It was proved in  Proposition    \ref{pr-1}  that a compact singular minimal surface with boundary contained in a plane orthogonal to $\vec{a}$ lies in one side of the boundary plane. We generalize this result in case that the boundary plane is arbitrary.  

\begin{proposition}\label{prl} Let $M$ be a compact singular minimal surface whose boundary $\Gamma$ is contained in a plane $\Pi$. 
\begin{enumerate}
\item If $\Pi$ is a vertical plane, then $M$ is a subset of $\Pi$.
\item If $\Pi$ is not a vertical plane, then $\mbox{int}(M)$ lies in one side of $\Pi$  and $M$ is not tangent  to $\Pi$ at any boundary point. More precisely, if $\alpha>0$ (resp. $\alpha<0$), then $\mbox{int}(M)$ lies below (resp. above) $\Pi$.
\end{enumerate}
\end{proposition}

\begin{proof}
If $\Pi$ is a vertical plane, then we use Corollary   \ref{prv2} obtaining the item (1). 

For item (2), we only do the proof for the case $\alpha>0$ (similarly if $\alpha<0$). Suppose        $\Pi=\{q\in\r^3: \langle q-q_0,\vec{v}\rangle=0\}$ where $\vec{v}\in\r^3$, $|\vec{v}|=1$ and $\langle \vec{v},\vec{a}\rangle>0$. Let $\Pi_t$ be the plane of equation $\langle q-q_0,\vec{v}\rangle=t$ and consider the foliation $\{\Pi_t:t\in\r\}$ of $\r^3$ by  parallel planes: notice that $\Pi_0=\Pi$. We denote by $\Pi_t$ again the part of $\Pi_t$ in the halfspace $\r^3_{+}$ and take  on $\Pi_t$ the orientation given by $\vec{v}$. Then  the weighted mean curvature $H_\varphi^t$ of $\Pi_t$ is  
$$H_\varphi^t(q)=-\frac{\alpha}{2}\frac{\langle \vec{v},\vec{a}\rangle}{\langle q,\vec{a}\rangle}, \quad (q\in \Pi_t)$$ 
and   $H_\varphi^t<0$ because $\alpha>0$ and $\langle\vec{v},\vec{a}\rangle>0$. For $t$ sufficiently large, we have $\Pi_t\cap M=\emptyset$  because $M$ is compact. Let us move  $\Pi_t$ towards $M$ by letting $t\searrow 0$. 

{\it Claim.} The plane $\Pi_t$ does not meet $M$ for every $t>0$.

The proof of the claim is by contradiction. Let $t_0>0$ be the first time that $\Pi_{t_0}$ touches $M$ at some point $p$. Since $\Gamma\subset\Pi_0=\Pi$, then $p\not\in\Gamma$, hence $p$ is an interior point and   $\Pi_{t_0}$ and $M$ are tangent at $p$. We know that $H_\varphi=0$ on $M$, independently on the orientation $N$. If we choose $N$  such that $N(p)$ coincides with $\vec{v}$, then $\Pi_{t_0}$ lies above $M$ around   $p$ and the tangency principle gives a contradiction. This proves the claim.

From the claim, and letting $t\searrow 0$,  we arrive until $\Pi_0$ and, consequently, $M$ lies  below the plane $\Pi$. The same  argument comparing now $M$ with the very plane $\Pi$ proves that $\Pi$ does not contain interior points of $M$. We conclude that  $\mbox{int}(M)$ lies below $\Pi$ and is   not    tangent to $\Pi$ at some boundary point (boundary version of the tangency principle).
\end{proof}

As a consequence of Proposition    \ref{prl}, we   obtain two types of  results answering to the Problem B   in case that   the singular minimal surface is embedded. The technique that we use is the  Alexandrov reflection method (\cite{al}). The key point of this technique  is that  by means of reflections about a uniparametric family of planes, we can compare the given surface with itself by means of  the tangency principle. The first result   says that a singular minimal surface   inherits some symmetries from its boundary.

\begin{theorem}\label{ts1} Let $\Gamma\subset \r^3_{+}$ be a simple closed curve contained in a non-vertical  plane $\Pi$. Assume that $\Gamma$ is symmetric about the reflection across a vertical  plane $P$  and that   $P$ separates  $\Gamma$ in  two graphs on the straight-line $\Pi\cap P$. If $M$ is an   embedded compact singular minimal surface with boundary $\Gamma$, then   $P$ is symmetry plane of $M$.
\end{theorem}

\begin{proof} By Proposition    \ref{prl}   we know that $\mbox{int}(M)$ lies in one side of  $\Pi$. If $\Omega\subset\Pi$ is the domain bounded by $\Gamma$, the  embeddedness of $M$ ensures that  $M\cup\Omega$ defines a closed surface without boundary in $\r^3$,  possibly  non-smooth along $\Gamma$. Hence, $M\cup\Omega$ determines a bounded $3$-domain   $\mathcal{D}\subset\r^3$. Then we  use the Alexandrov reflection method by planes parallel to $P$.   Let us observe that the plane $\Pi$ is not necessarily horizontal (i.e. orthogonal to the direction $\vec{a}$).  After a rotation about the $z$-axis and up to a horizontal translation, we suppose that $P$ is the plane of equation $x=0$.

The proof is by contradiction. Suppose that $P$ is not a plane of symmetry of $M$, and thus    there are two points $q_1, q_2\in M\setminus\Gamma$ such that the line $\overline{q_1q_2}$ joining $q_1$ to $q_2$ is orthogonal to $P$, $q_1$ and $q_2$ are in opposite sides of $P$ and $\mbox{dist}(q_1,P)>\mbox{dist}(q_2,P)$. Without loss of generality, suppose $x(q_2)<0<x(q_1)$. Then the symmetric point of $q_1$ about $P$, say $q_1^*$, satisfies $x(q_1^*)<x(q_2)$.  For any $t\in\r$, denote $P_t$ the plane of equation $x=t$. Let  $M(t)^{-}=M\cap \{x\leq t\}$, $M(t)^{+}=M\cap \{x\geq t\}$ and $M(t)^*$ the reflection of $M(t)^{+}$ about $P_t$. We notice that the reflection about $P_t$ preserves the singular minimal surface property as well as this reflection leaves invariant as a subset, the boundary plane $\Pi$. Since $M$ is compact, we have $P_t\cap M=\emptyset$ for $t$ sufficiently large. Then we move $P_t$ towards $M$ by letting $t\searrow 0$, until the first contact point with $M$ at the time $t_1>0$. Then we move slightly more $P_{t_1}$, $t<t_1$, and we reflect $M(t)^{+}$.  The embeddedness of $M$, and being $M$   below $\Pi$, assures the existence of $\epsilon>0$ such that $\mbox{int}(M(t)^*)\subset \mathcal{D}$ for every $t\in (t_1-\epsilon,t_1)$. Because $M$ is compact,  there is $t_2\geq 0$ with $t_2<t_1$, such that $\mbox{int}(M(t)^*)$ is outside $\mathcal{D}$ for any $t<t_2$. In fact, $t_2>0$ by the existence of the points $q_1$ and $q_2$ and $0<x(q_1^*)<x(q_2)$. Furthermore, and because $M(t)^{+}$ is a graph of $P_t$ for $t>t_2$ and $\Gamma\cap\{x>0\}$ is also a graph on the line $P\cap \Pi$, we have $\partial M(t_2)^*\cap \Gamma\subset P_{t_2}$. There are two possibilities:
\begin{enumerate}
\item There exists $p\in \mbox{ int}(M(t_2)^*)\cap  \mbox{ int} (M(t_2)^{-})$. Because $p$ is an interior point, $M(t_2)^*$ and $M(t_2)^{-}$ are tangent at $p$. Since $M(t_2)^*$ is in one side of $M(t_2)^{-}$, and both surfaces are singular minimal surfaces, the tangency principle and the analyticity of $M$ imply that $M(t_2)^*=M(t_2)^{-}$, and thus,  $P_{t_2}$ is a plane of symmetry of $M$: a contradiction because $t_2>0$ and $\Gamma$ is not invariant by reflections across $P_{t_2}$.
\item The surface $M$ is orthogonal to $P_{t_2}$ at some point $p\in\partial M(t_2)^{*}\cap \partial M(t_2)^{-}$. This point $p$ can not belong to $\Gamma$, that is, $p\not\in \Gamma\cap P_{t_2}$ because $\Gamma$ is a graph on $P_{t_2}\cap \Pi$.  We now use the boundary version of the tangency principle, concluding that $P_{t_2}$ is a plane of symmetry of $M$, a contradiction again.
\end{enumerate}
\end{proof}

Theorem \ref{ts1}  may not be true in case that $P$ separates $\Gamma$ in two symmetric pieces that are not graphs on $\Pi\cap P$: it has been crucial in the above proof to prevent   that the contact point $p$ could belong to $(\partial M(t_2)^*\cap\Gamma)\setminus P_{t_2}$. 

 In the special case that $\Gamma$ is a circle contained in a horizontal plane $\Pi$, we deduce that  $M$ is symmetric about every plane orthogonal to $\Pi$ and containing  the center of $\Gamma$. 

 \begin{corollary}\label{c-c1}
  The only  embedded compact singular minimal surfaces with circular boundary contained in a horizontal plane are   surfaces of revolution whose rotation axis is vertical.
 \end{corollary}

\begin{remark}When the boundary  is a circle contained in a non-horizontal plane, it is expectable that the surface, if exists,  is not rotational because the surface contains  a (full) circle but   the rotation axis must be vertical, which is not possible. In fact,  we find explicit  examples. If $\alpha<0$, it was proved in \cite{lo2} the solvability of the Dirichlet problem for the equation (\ref{mean}) on convex domains and arbitrary boundary data. So, let $C$ be a circle contained in a tilted plane $P$ and let $D$ be the projection on $\r^2$ of the round disk bounded by $C$ in $P$. Since $D$ is convex, there is an $\alpha$-singular minimal graph $M$ on $D$ whose boundary is $C$, but  $M$ can not be  a surface of revolution about a vertical axis. 
\end{remark}
Thanks to Corollary \ref{c-c1}, we   rewrite   Theorem  \ref{t-esti}.  

\begin{corollary} Let $\alpha>0$. Let $\Gamma$ be a circle of radius $r>0$ contained in the horizontal plane of equation $z=c>0$.  If $M$ is an  embedded compact $\alpha$-singular minimal surface with boundary $\Gamma$, then  the lowest point $p_0$ of $M$, satisfies 
$$z(p_0)<\left(\frac{\alpha+2}{\alpha}\right)c-\sqrt{r^2+\frac{4c^2}{\alpha^2}}.$$
\end{corollary}

We prove that the assumption of embeddedness  in   Corollary   \ref{c-c1}  can be dropped if $\alpha$ is negative. 
 
 \begin{proposition}\label{pr-ne}
  Let $\alpha<0$. Let $\Gamma\subset\r^3_{+}$ be a simple  closed convex surface contained in a horizontal plane $\Pi$. If $M$ is a compact $\alpha$-singular minimal surface with boundary $\Gamma$, then $M$ is a graph on $\Pi$. As a consequence, if $\Gamma$   is a circle, then $M$ is a surface of revolution.
 \end{proposition}
 
 \begin{proof}  Denote by $\Omega\subset\Pi$ the domain bounded by $\Gamma$  and let $\Omega^*=\pi(\Omega)\subset\r^2\times\{0\}$, where $\pi$ is the  orthogonal projection on the $xy$-plane. Because $\Omega^*$ is convex, it follows from  Corollary \ref{prv2} and Proposition \ref{pr-1} that $\mbox{int}(M)\subset\Omega^*\times\r$ and $M$ lies above the plane $\Pi$. By contradiction, suppose that $M$ is not a graph on $\Omega^*$. Then there exist two points of $M$, namely, $q_1$ and $q_2$, such that $\pi(q_1)=\pi(q_2)$. Let $q_0=\pi(q_1)\in \Omega^*$.
 
By dilations $\lambda M$, $\lambda>1$, with center the point $q_0$, we move upwards $M$ until that $\lambda M$ does not intersect $M$. Then we come back  by letting $\lambda\searrow 1$. Since the vertical line starting at $q_0$ meets at least twice $M$, namely, at $q_1$ and $q_2$, there  is $\lambda_1>1$ such that $\lambda_1 M$ intersects the first time $M$. Let $p\in M\cap \lambda_1 M$. Because $M\subset\Omega^*\times\r$ and $\Gamma$ is convex,  the boundary curve $\partial (\lambda_1 M)=\lambda_1\Gamma$ lies outside of $\Omega^*\times\r$. In fact, $\partial(\lambda_1M)$ is contained in the halfcone $\{q_0+t(\gamma-q_0):t>0, \gamma\in\Gamma\}$. This implies that $p$ is a common interior point of $M$ and $\lambda_1 M$. As $\lambda_1 M$ lies in one side of $M$ around $p$,   the tangency principle implies that $M$ is a subset of $\lambda_1 M$: a contradiction, because $\partial M\not=\partial(\lambda_1 M)$. This proves that $M$ is   a graph on $\Omega^*$.

For the second statement, we know that    the surface is a graph, in particular, it is embedded and we   apply Corollary \ref{c-c1}. 
 \end{proof}

 Our second result replaces the symmetry of the boundary curve in Theorem \ref{ts1} by the constancy of the angle between the surface and the  boundary plane.

 \begin{theorem}\label{ts2}
 Let $\Pi$ be a non-vertical   plane and let  $M$ be an  embedded compact singular minimal surface with $\partial M\subset\Pi$. If $M$ makes a constant contact angle with $\Pi$ along $\partial M$, then  $M$ has a symmetry about a vertical plane.
  \end{theorem}

  \begin{proof} The proof uses again the Alexandrov reflection method. We know by Proposition    \ref{prl} that the interior of $M$ lies contained in one side of  $\Pi$ and $M$ is not tangent to $\Pi$.    With  the same notation as in Theorem  \ref{ts1}, we only indicate  the differences. Let $\vec{w}$ be a horizontal vector and parallel to $\Pi$: if $\Pi$ is not a horizontal plane, this  direction $\vec{w}$ is unique. Let $\{P_t\}_{t\in\r}$ be the foliation of $\r^3$ by planes orthogonal to $\vec{w}$ which, after a rotation about the $z$-axis and a horizontal translation, we suppose that $P_t$ is the plane of equation $x=t$ (and $\vec{w}=e_1$).
  
  We begin with the reflection method as in Theorem  \ref{ts1}. After the first time $t=t_1$ where   $P_{t_1}$ touches $M$, we arrive until $t=t_2$ where $\mbox{int}(M(t)^*)$ is  outside $W$ for every $t<t_2$. There  possibilities are now:
  \begin{enumerate}
  \item There exists $p\in \mbox{int}(M(t_2)^*)\cap \mbox{int}(M(t_2)^{-})$.
  \item $M$ is orthogonal to $P_{t_2}$ at some point $p\in \partial M(t_2)^*\cap \partial M(t_2)^{-}\setminus\Gamma$.
    \item There exists $p\in\partial M(t_2)^*\cap\partial M(t_2)^{-}\cap \Gamma$ and $p\not\in P_{t_2}$.
   \item $M$ is orthogonal to $P_{t_2}$ at some point $p\in \partial M(t_2)^*\cap \partial M(t_2)^{-}\cap\Gamma$.

   \end{enumerate}

The cases (1) and (2) appeared in Theorem  \ref{ts1} and the tangency principle  implies  that $P_{t_2}$ is a plane of symmetry of $M$.  In case (3), we have $T_p\partial M(t_2)^*=T_p\partial M(t_2)^{-}$ and the surfaces $M(t_2)^*$ and $M(t_2)^{-}$ are tangent at $p$ because the contact angle between $M(t_2)^*$ and $P_{t_2}$ coincides with the one between $M(t_2)^{-}$ and $P_{t_2}$: here we use that $P_{t_2}$ is a vertical plane and that $\Pi$ is invariant by reflection about $P_{t_2}$. Then we  apply the boundary version of the tangency principle to conclude that   $P_{t_2}$ is a plane of symmetry of $M$. In case (4), the tangency principle at a corner point (\cite{se}) proves that $M(t_2)^{-}=M(t_2)^*$ and thus, $P_{t_2}$ is a plane of symmetry of $M$ again.
   \end{proof}
   
As in Corollary   \ref{c-c1}, we conclude:
  \begin{corollary}\label{c-c2} Let  $M$ be an  embedded compact singular minimal surface whose boundary is contained in a horizontal plane $\Pi$.   If $M$ makes a constant contact angle with $\Pi$ along $\partial M$, then  $\partial M$ is a circle and $M$ is a surface of revolution with vertical  rotation axis.
 \end{corollary}

Theorems \ref{ts1} and \ref{ts2} extend when the weighted mean curvature is constant, but we need to assume that the surface is contained in one side of the boundary plane: for singular minimal surfaces, this property  is assured by Proposition    \ref{prl}. So, we  prevent that  during the reflection process,   the first contact point between the reflected surface and the initial surface (for $t=t_2$) may  occur between   interior  and    boundary points, so we can not  apply the tangency principle. In order to simplify the statements, we only give the result when $\Pi$ is a horizontal plane and $\Gamma$ is a circle in Theorem  \ref{ts1}.

\begin{corollary}  Let  $M$ be an  embedded compact surface with constant weighted mean curvature. Suppose one of the next assumptions:
\begin{enumerate}
\item The boundary of $M$ is a circle contained in a horizontal plane.
\item The boundary of $M$ is   is contained in a horizontal plane $\Pi$ and  $M$ makes a constant contact angle with $\Pi$ along $\partial M$.
\end{enumerate}
If $M$ is contained in one side of the boundary plane, then  $M$ is a surface of revolution about a vertical axis.
\end{corollary}
 \begin{proof} The proof runs with similar arguments.  We only need  the next observation. For applying the tangency principle, one requires that the orientations agree at the contact point $p$ between $M(t_2)^*$ and $M(t_2)^{-}$. For this,  we orient $M$ so $N$ points towards the domain $\mathcal{D}$. Then the Gauss map of $M(t_2)^*$ at   $p$ points towards $\mathcal{D}$ again, so it coincides with  $N$ at $M(t_2)^{-}$ and the tangency principle applies.
 \end{proof}
 
   \section{Non-existence results of compact singular minimal surfaces spanning two curves}\label{s-b2}

 For $\alpha>0$,   we show in this section  results of non-existence of  compact connected $\alpha$-singular minimal surfaces     spanning two given curves $\Gamma_1$ and $\Gamma_2$. As a general conclusion, if $\Gamma_1$ and $\Gamma_2$  are sufficiently far apart, it is not possible that $\Gamma_1$ and $\Gamma_2$ can bound a compact connected singular minimal surface. The   value of this separation  depends not only on the distance  between $\Gamma_1$ and $\Gamma_2$, but also on the height of these curves with respect to the plane $\r^3_0(\vec{a})$. We will obtain   two  types   of results depending if the separation is   along the direction of $\vec{a}$, or with respect to  an orthogonal direction to $\vec{a}$.    After a change of coordinates, we suppose   that the density vector is $\vec{a}=(0,0,1)$.
  
 We need  the next definitions. A   {\it slab} $W\subset\r^3$  is the domain  between two   parallel planes $P_1$ and $P_2$. The   width of a slab is  the distance between $P_1$ and $P_2$. If $\Gamma_1$ and $\Gamma_2$ are two closed curves, we say that a slab $W$ separates $\Gamma_1$ from $\Gamma_2$ if $\Gamma_1$ and $\Gamma_2$ are contained  in different components of $\r^3\setminus W$.
On the other hand, the  {\it height} of a bounded set $A\subset\r_{+}^3$ is   $h(A)=\sup\{z(p):p\in A\}$. 

In the proofs, we will use other singular minimal surfaces as barriers in a comparison argument.   The first surfaces to use are catenary cylinders. We  indicate the main properties of the $\alpha$-catenaries: see details in \cite{lo}.  

 \begin{proposition}\label{p-cate} Let $\alpha>0$. Let $\gamma:I\rightarrow\r^2$ be an $\alpha$-catenary with respect to the vector $(0,1)$.  Then, after a horizontal translation, $\gamma$ is the graph of a function $z=f(x) $, $f:(-R,R)\rightarrow\r^+$, that  satisfies  
 \begin{eqnarray}
 &&\frac{f''(x)}{1+f'(x)^2}=\frac{\alpha}{f(x) },\label{cate2}\\
&& f(0)=z_0,\quad f'(0)=0,\label{cate3}
 \end{eqnarray}
 and $f$ is an even convex  function with a unique minimum at $x=0$. Here $(-R,R)$ is the maximal domain.  Furthermore, 
   \begin{enumerate}
 \item If $\alpha>1$, then $R<\infty$ and $\lim_{x\rightarrow\pm R}f(x)=\infty$. The function $R=R(z_0)$ depending on  $z_0$ is strictly increasing with $\lim_{z_0\rightarrow 0}R=0$ and $\lim_{z_0\rightarrow\infty} R=\infty$.
 \item If $\alpha\in (0,1]$, then $R=\infty$.  Fix $x_0>0$ and let $f(x_0)=f(x_0;z_0)$ be the value of $f$ at $x_0$ as a function on   $z_0$, $z_0$  the initial condition in (\ref{cate3}). Then there exists $z_0(\alpha)$ such that $f(x_0)$ attains a minimum   at  $z_0(\alpha)$, $f(x_0)$ is decreasing in the interval $(0,z_0(\alpha))$, increasing in the interval  $(z_0(\alpha),\infty)$ and  $\lim_{z_0\rightarrow 0}f(x_0)=\lim_{z_0\rightarrow \infty}f(x_0)=\infty$.
     \end{enumerate}
     \end{proposition}

The item (1) says that if $\alpha>1$, then the graph of $f$ is asymptotic to the vertical lines $x=\pm R$. The item (2) generalizes the case of the catenary. Indeed, if  $\alpha=1$, the solution   of (\ref{cate2})-(\ref{cate3}) is the catenary $f(x)=z_0\cosh(x/z_0)$ and we can find the value  $z_0(\alpha)$.  It is not difficult  to see that the minimum of $f(x_0;z_0)$  corresponds with the  (unique) positive root $z_0$ of the equation
$$
\tanh\left(\frac{x_0}{z_0}\right)=\frac{z_0}{x_0}
$$
 and  $f(x_0;z_0(\alpha)) \simeq 1.5088\ x_0$. We now prove the first result of non-existence. 

        \begin{theorem}\label{t-n1} Let $\alpha>0$ and  $m>0$. There exists $h_0=h_0(m,\alpha)>0$ depending only on $m$ and $\alpha$,  such that if  $\Gamma_1$ and $\Gamma_2$ are two   closed curves in $\r^3_+$ and  there is a vertical slab of width $m$ separating $\Gamma_1$ from $\Gamma_2$ with $h(\Gamma_1\cup\Gamma_2)\leq h_0$, then there does not  exist a  compact connected $\alpha$-singular minimal surface with boundary $\Gamma_1\cup\Gamma_2$.
     \end{theorem}

    \begin{proof} The proof is by contradiction.     We distinguish two cases  depending on  $\alpha$. 
    
    {\it  Case $\alpha  \in (0,1]$}.  From Proposition    \ref{p-cate}, we know that the value $x_0=m/2$ belongs to the maximal domain of any $\alpha$-catenary. Following the notation of Proposition    \ref{p-cate}, consider  $z_0(\alpha)>0$ and set  
    $$h_0=h_0(m,\alpha)=f(m/2;z_0(\alpha)).$$
Suppose that $M$ is a  compact connected $\alpha$-singular minimal   surface with $h(\Gamma_1\cup\Gamma_2)\leq h_0$ and, after a rotation about the $z$-axis and a horizontal translation, we suppose that the vertical slab $W$   separating $\Gamma_1$ from $\Gamma_2$   is defined by the planes  $P_1$ and $P_2$ of equations $x=r$ and $x=-r$, respectively, with $r=m/2$. Let $\gamma$ be  the graphic of the $\alpha$-catenary    in the $xz$-plane  that is the solution of (\ref{cate2})-(\ref{cate3})  and the  catenary cylinder $S_\gamma=\{\gamma(s)+t e_2: s\in (-R,R), t\in\r\}$.      
    
We use  dilations of $S_\gamma$  from the origin $O$, $\lambda>0$. Notice that $\lambda S_\gamma$ is the catenary cylinder based on $\lambda\gamma$, namely,  the solution of  (\ref{cate2}) when $f(0)= \lambda z_0(\alpha)$: see Figure  \ref{fig4}, left.         Recall that the lowest points of $S_\gamma$ is  the horizontal line $L$ of equation $\{x=0, z=z_0(\alpha)\}$ because $\gamma$ is symmetric about the $z$-axis in the $xz$-plane. Let  $\widetilde{\lambda S_\gamma}$ be the part of $\lambda S_\gamma$ contained in the slab $W$ whose boundary  consists of two parallel horizontal lines.     Thus the height of the lowest points of $\widetilde{\lambda S_\gamma}$ is $\lambda z_0(\alpha)$. For $\lambda$ sufficiently large, this height is bigger than $h_0$ and thus  $\widetilde{\lambda S_\gamma}\cap M=\emptyset$. Let $\lambda\searrow 0$. Since the plane of equation $x=0$ divides $W$ in two components and $\lambda z_0(\alpha)\rightarrow 0$, there exists a first contact with $M$ at $\lambda=\lambda_0$, that is, $\widetilde{\lambda S_\gamma}\cap M=\emptyset$ for $\lambda>\lambda_0$ and $\widetilde{\lambda_0 S_\gamma}\cap M\not=\emptyset$. On the other hand, the height  of the boundary of $\widetilde{\lambda_0 S_\gamma}$ is bigger than $h_0$, hence the contact point $p$ between $M$ and $\widetilde{\lambda_0 S_\gamma}$ must an interior point. See Fig \ref{fig4}, right. Because $M$ lies in one side of $\widetilde{\lambda_0 S_\gamma}$ around $p$,  the tangency principle implies $M\subset \widetilde{\lambda_0 S_\gamma}$: this is a contradiction because $\partial M=\Gamma_1\cup\Gamma_2$ is outside $W$.

{\it  Case $\alpha>1$}.   By Proposition    \ref{p-cate}, let $z_0>0$ be the unique number such that $R(z_0)=m/2$ and let 
$$h_0=h_0(m,\alpha)=z_0.$$
Let $M$ be a  compact connected $\alpha$-singular minimal   surface with $h(\Gamma_1\cup\Gamma_2)\leq h_0$. We follows the same steps than in the  case $\alpha\in (0,1]$.  By the definition of $R(z_0)$,  the catenary cylinder $S_\gamma$ is asymptotic to the planes $P_1$ and $P_2$. Because $h(\Gamma_1\cup\Gamma_2)\leq h_0$, Proposition    \ref{prl} asserts that $M$ lies  below the plane of equation $z=h_0$, in fact, $M\cap S_\gamma=\emptyset$: if there is an intersection point $p\in M\cap S_\gamma$, then necessarily $p$ is a point with the lowest height of $S_\gamma$ being $M$ and $S_\gamma$ tangent at that point, which is impossible by the tangency principle.
 
        Now we consider the dilations $\lambda S_\gamma$ of $S_\gamma$.    Letting $\lambda\searrow 0$,   $\lambda S_\gamma$ have the property that the height of  the lowest points goes to $0$, $\lambda S_\gamma$ is asymptotic to two vertical planes parallel to $P_i$  and both planes are contained in the slab $W$ because $\lim_{\lambda\rightarrow 0}R(\lambda z_0)=0$.  Then there is a first number $\lambda=\lambda_1$ such that $\lambda_1 S_\gamma$ intersects $M$ at some interior point: a contradiction again by  the tangency principle.
          \end{proof}

\begin{figure}[hbtp]
\begin{center}
\begin{center}\resizebox{\textwidth}{!}{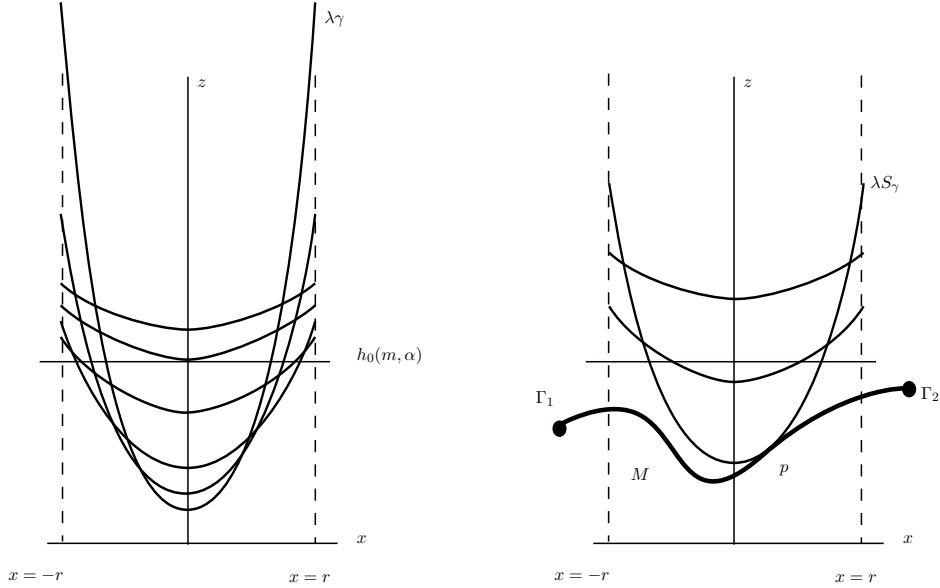}\end{center}
\end{center}
\caption{Case $\alpha=1$. Left: the catenaries $\lambda \gamma$. Right: the contact point between $M$ and $\lambda S_{\gamma}$}\label{fig4}
\end{figure}

\begin{remark}
We point out the existence of non-compact rotational  surfaces spanning  two curves which are separated by a vertical slab. The next example is for  $\alpha=1$ and $\vec{a}=(0,0,1)$. Following   \cite{lo}, there exists a   rotational $1$-singular minimal surface about the $x$-axis  where the generating curve $\gamma$  is   a $2$-catenary (see Figure  \ref{figex3}).  Then   $\gamma(x)=(x,0,z(x))$, $x\in (-R,R)$ and $z(x)$ satisfies (\ref{cate2})-(\ref{cate3}) for $\alpha=2$. If $M_\gamma$ is the surface obtained by rotating $\gamma$, then  $\overline{M_\gamma}\cap\r^3_0\not=\emptyset$, indeed, 
$$\overline{M_\gamma}\cap\r_0^3=\{(x,z(x),0):x\in(-R,R)\}\cup\{(x,-z(x),0): x\in(-R,R)\}.$$
 Fix $0<r<R$ and consider the part of $M_\gamma$ included in  the vertical slab $W$ of equation $-r\leq x\leq r$. Then  $\partial M_\gamma$ is formed by two half-circles in the planes $x=\pm r$ centered at the points $(\pm r,0,0)$ and  radii  $z(r)$. Here  $h(\partial M_\gamma)=z(r)$. However if we compare with the catenary cylinders utilized in the proof of Theorem  \ref{t-n1}, the first contact point  with $M_\gamma$ occurs at some boundary point of $M_\gamma$ and the tangency principle does not apply. This is because $h(\partial M_\gamma)>h_0(m)$ where $r=m/2$. A numerical example is the following.  Consider  $z_0=1$ in (\ref{cate3}). Then $R\simeq1.31103$. If we take $r=1$, then $z(1)\simeq3.21815$. On the other hand, the value $h_0(1)$  is $f(1)\simeq	1.50880$.
\end{remark}

The following result   refers to the special case that the boundary of the surface is formed by two curves contained  in different horizontal planes and separated by a vertical plane.   The following result holds for any $\alpha$.

  \begin{theorem}\label{t-ros}   Let  $\Gamma_1$ and $\Gamma_2$ be two closed curves in $\r^3_{+}$ contained in different horizontal planes. If there exists a vertical plane separating $\Gamma_1$ from $\Gamma_2$, then there does not exist a compact connected $\alpha$-singular minimal surface with boundary $\Gamma_1\cup\Gamma_2$.
     \end{theorem}
\begin{proof} We follow the same ideas as in \cite{ro} in the context of  minimal surfaces. The proof is by contradiction. Let $M$ be a  compact connected singular minimal surface  spanning $\Gamma_1\cup\Gamma_2$ and let $P$ be a vertical plane separating $\Gamma_1$ from $\Gamma_2$.  After a horizontal translation and a rotation about the $z$-axis, we suppose that   $P$ is the plane of equation $x=0$ with $\Gamma_1\subset \{x>0\}$ and $\Gamma_2\subset \{x<0\}$. Let $\mathcal{R}$ be the reflection  about $P$. Then $\mathcal{R}(M)$ is a singular minimal surface whose boundary is $\mathcal{R}(\Gamma_1)\cup \mathcal{R}(\Gamma_2)$ and $\mathcal{R}(\Gamma_1)\subset\{x<0\}$ and $\mathcal{R}(\Gamma_2)\subset \{x>0\}$.
We translate $\mathcal{R}(M)$ in the direction of the $y$-axis until that $\mathcal{R}(M)$ is disjoint from $M$ and then come back it until the first time when $\mathcal{R}(M)$ touches $M$ at some point $p$. Because $P$ separates $\Gamma_1$ from $\Gamma_2$ and $\mathcal{R}(\Gamma_1)$ from $\mathcal{R}(\Gamma_2)$, this point $p$ must be an interior point. The tangency principle assures that $\mathcal{R}(M)=M$, which is a contradiction because $\partial \mathcal{R}(M)\not=\partial M$.
\end{proof}

 The following  non-existence result asserts that if two curves  are sufficiently far apart in vertical distance then they can not bound a  compact connected singular minimal surface. This result  is similar to what happens with the catenoid:  if   two coaxial circles are sufficiently close, there exists a catenoid spanning both circles, but if we  go separating the two circles sufficiently far, there is a time where  the catenoid breaks. In order to state our result,  we will fix one boundary curve and we change the position of the second one. Firstly, we suppose that $\Gamma_1$ and $\Gamma_2$ are contained   in a horizontal plane.
 
   \begin{theorem}\label{t-ne} Let $\alpha>0$. Let $\Gamma_1\subset\r^3_+$ be a closed curve contained in the   plane $P_1$ of equation $z=c_1$. Let  $\Gamma_2\subset P_1$ be a closed curve and denote $\Gamma_2^t$  the vertical translation of $\Gamma_2$ situated in the plane of equation $z=t$. Then   there exists $d_0=d_0(\alpha, \Gamma_1,\Gamma_2,c_1)$   such that if $t>d_0$, then there does not exist  a   compact connected $\alpha$-singular minimal surface with boundary $\Gamma_1\cup\Gamma_2^t$. 
     \end{theorem}

\begin{proof} Let $R>0$ be a number such that $\Gamma_1$ and $\Gamma_2$ are included in the interior of the disk $D\subset P_1$ of radius $R$ centered at $(0,0,c_1)$. Let $c=1+c_1$. We will  work with the winglike-shaped $\alpha$-singular minimal surfaces    (\cite{ke}).  If $\gamma=(x(s),0,z(s))$ is the generating curve of this surface   parametrized by the arc length, then   $x'(s)=\cos\theta(s)$, $z'(s)=\sin\theta(s)$ for some function $\theta$. Now  equation (\ref{mean}) is   
\begin{equation}\label{angle}
\theta'(s)+\frac{\sin\theta(s)}{x(s)}=\alpha\frac{\cos\theta(s)}{z(s)}.
\end{equation}
Consider  initial conditions 
$$x(0)=\lambda,\quad z(0)=c>0,\quad  \theta(0)=0$$
and denote the solution $\{x_\lambda(s),z_\lambda(s),\theta_\lambda(s)\}$ to indicate its dependence on the parameter $\lambda$.  If $\gamma_\lambda(s)=(x_\lambda(s),0,z_\lambda(s))$, then   $\gamma_\lambda$ has a horizontal tangent line at the point $(\lambda,0,c)$. Moreover, this is the only point of $\gamma$ where the tangent line is horizontal and  also the  lowest point of $\gamma$, so the circle that generates on the surface lies in the lowest position. See Figure  \ref{figwing}.

\begin{figure}[hbtp]
\begin{center}
\begin{center}\resizebox{\textwidth}{!}{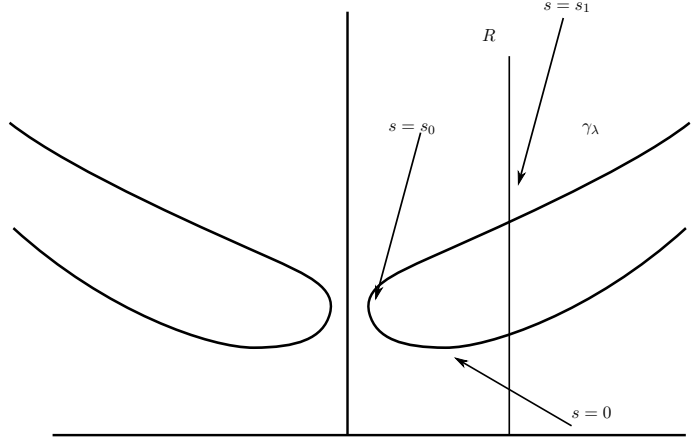}\end{center}
\end{center}
\caption{The generating curve $\gamma_\lambda$ for a winglike singular minimal surface  and the points $s=0$, $s=s_0$ and $s=s_1$. Here $\alpha=1$, $\lambda=1$ and $c=2$}\label{figwing}
\end{figure}

Recall that  the graphic of $\gamma_\lambda$ is asymptotic to the line $\{y=0,z=\sqrt{\alpha}x\}$ and  there exists $s_0<0$ such that $\theta(s_0)=-\pi/2$, that is, the tangent line of $\gamma$ at $s=s_0$ is vertical (\cite{ke}). This value $s_0$ determines the waist $C(x(s_0))\times\{z(s_0)\}$ of the rotational surface, where $C(x(s_0))$ is the circle of radius $x(s_0)$ on the $xy$-plane  centered at the origin.

As $\lambda\rightarrow 0$, the graphic of $\gamma_\lambda$ converges to the generating curve   $\gamma_0$ that intersects orthogonally the $z$-axis at the point $z=c$. See Figure  \ref{figesti}, left. For each $\lambda>0$, let $s_1=s_1(\lambda)<s_0$ be the unique value such that $x_\lambda(s_1)=R$. Let us observe that if $\lambda\rightarrow 0$, then $s_1\rightarrow 0$ and if $\lambda$ is sufficiently large, then $s_1$ does not exist: this situation occurs at least   when $x_\lambda(s_0)>R$. Let
$$d_0=\sup\{z_\lambda(s_1(\lambda)): \lambda>0\}.$$
With this value of $d_0$, we prove the  theorem by contradiction. Let $M$ be a    compact connected $\alpha$-singular minimal surface with boundary $\Gamma_1\cup\Gamma_2^t$ and  suppose $t>d_0$. It follows by   Corollary   \ref{prv2} that  $M$ is contained in the solid cylinder $D\times\r$. Take $\{S_\lambda:\lambda>0\}$ the uniparametric family of rotational $\alpha$-singular minimal surfaces whose generating curves are $\gamma_\lambda$: see Figure  \ref{figesti}, left. Let $\lambda_0$ be sufficiently large so $S_{\lambda_0}\cap M=\emptyset$. By letting  $\lambda\searrow 0$,   we know that the surfaces $S_\lambda$ does not intersect the disk $D\times\{c_1\}$ because $c>c_1$, neither  the disk $D\times\{c_2\}$ because $z_\lambda(s_1)\leq d_0<c_2$. The waist $C(x_\lambda(s_0))\times \{z_\lambda(s_0)\}$ tends to the point $(0,0,c)$ as $\lambda\rightarrow 0$. Since $M$ is connected,   there would be a first value $\lambda=\lambda_1$ such that $S_{\lambda_1}$ has a contact point $p$ with $M$, which must be an interior point: see Figure  \ref{figesti}, right. The tangency principle implies   $M\subset S_{\lambda_1}$, a contradiction because $\partial M$ is not contained in $S_{\lambda_1}$. 
\end{proof}

\begin{figure}[hbtp]
\begin{center}
\begin{center}\resizebox{\textwidth}{!}{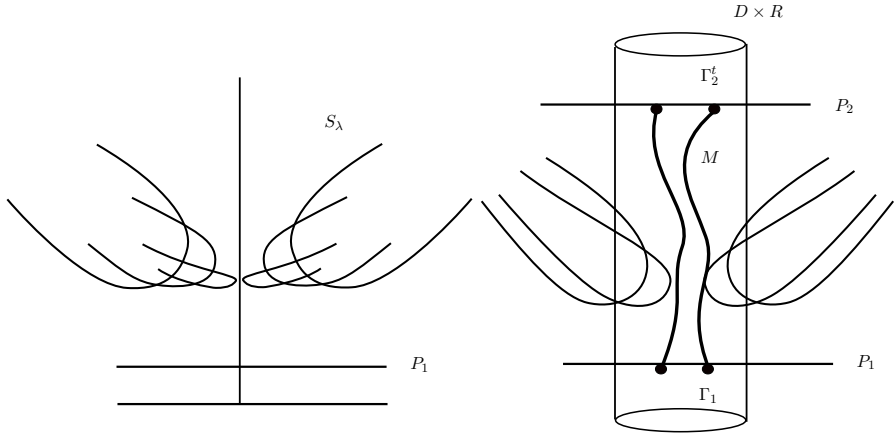}\end{center}
\end{center}
\caption{Left: the family of winglike singular minimal surfaces. Right: the first contact point between $S_\lambda$ and $M$}\label{figesti}
\end{figure}

As a consequence of this result,   we   consider the general case that $\Gamma_1$ and $\Gamma_2$ are not   contained in horizontal planes. 

\begin{corollary} 
Let $\alpha>0$. Let $\Gamma_1$ and $\Gamma_2$ be two closed curves of $\r^3_+$. Fix $\Gamma_1$ and let  $c_1=\sup\{z(p):p\in\Gamma_1\}$.  Then there exists $d_0>c_1$ such that if  $t>d_0$,  then there does not exist a   compact connected $\alpha$-singular minimal surface with boundary $\Gamma_1\cup\Gamma_2^t$.
\end{corollary}

\begin{proof}
Let $R>0$ be sufficiently large so the orthogonal projection of $\Gamma_i$ on the $xy$-plane is contained in a disk $D$ of radius $R$. If $P_1$ is the plane of equation $z=c_1$, let us move up vertically $\Gamma_2$ until the position $\Gamma_2^t$ such  that the value $c_2:=\inf\{z(p):p\in\Gamma_2^t\}$ satisfies  $c_2>d_0$, where $d_0$ is given by  Theorem \ref{t-ne}. Here  $P_2$ is the plane of equation $x=c_2$. If $M$ is a connected $\alpha$-singular minimal surface with boundary $\Gamma_1\cup\Gamma_2^t$, consider a component $M^*$ of $M\cap\{c_1\leq z\leq c_2\}$ such that its boundary $\partial M^*$ meets both $P_1$ and $P_2$. Let us observe that $\partial M^*$ may have many be not connected, but all their components are included  in the planes $P_1$ or $P_2$ and in the solid cylinder $D\times \r$. Then the same argument as in Theorem \ref{t-ne} gives a contradiction.  
\end{proof}


\end{document}

%% file: catenary.pdf_tex
\begingroup%
  \makeatletter%
  \providecommand\color[2][]{%
    \errmessage{(Inkscape) Color is used for the text in Inkscape, but the package 'color.sty' is not loaded}%
    \renewcommand\color[2][]{}%
  }%
  \providecommand\transparent[1]{%
    \errmessage{(Inkscape) Transparency is used (non-zero) for the text in Inkscape, but the package 'transparent.sty' is not loaded}%
    \renewcommand\transparent[1]{}%
  }%
  \providecommand\rotatebox[2]{#2}%
  \ifx\svgwidth\undefined%
    \setlength{\unitlength}{595.27559055bp}%
    \ifx\svgscale\undefined%
      \relax%
    \else%
      \setlength{\unitlength}{\unitlength * \real{\svgscale}}%
    \fi%
  \else%
    \setlength{\unitlength}{\svgwidth}%
  \fi%
  \global\let\svgwidth\undefined%
  \global\let\svgscale\undefined%
  \makeatother%
  \begin{picture}(1,0.7047619)%
    \put(0,0){\includegraphics[width=\unitlength,page=1]{catenary.pdf}}%
    \put(0.78208873,0.55117128){\color[rgb]{0,0,0}\makebox(0,0)[lb]{\smash{$z$}}}%
    \put(0.94979719,0.07165475){\color[rgb]{0,0,0}\makebox(0,0)[lb]{\smash{$x$}}}%
    \put(0.87836576,0.03141556){\color[rgb]{0,0,0}\makebox(0,0)[lb]{\smash{$x=r$}}}%
    \put(0.58352578,0.03257549){\color[rgb]{0,0,0}\makebox(0,0)[lb]{\smash{$x=-r$}}}%
    \put(0.91563435,0.4456193){\color[rgb]{0,0,0}\makebox(0,0)[lb]{\smash{$\lambda S_{\gamma}$}}}%
    \put(0,0){\includegraphics[width=\unitlength,page=2]{catenary.pdf}}%
    \put(0.20829564,0.55120188){\color[rgb]{0,0,0}\makebox(0,0)[lb]{\smash{$z$}}}%
    \put(0.37600412,0.07168534){\color[rgb]{0,0,0}\makebox(0,0)[lb]{\smash{$x$}}}%
    \put(0.30457269,0.03144641){\color[rgb]{0,0,0}\makebox(0,0)[lb]{\smash{$x=r$}}}%
    \put(0.00973268,0.03260625){\color[rgb]{0,0,0}\makebox(0,0)[lb]{\smash{$x=-r$}}}%
    \put(0.34184132,0.61826705){\color[rgb]{0,0,0}\makebox(0,0)[lb]{\smash{$\lambda\gamma$}}}%
    \put(0.37600412,0.26449791){\color[rgb]{0,0,0}\makebox(0,0)[lb]{\smash{$h_0(m,\alpha)$}}}%
    \put(0,0){\includegraphics[width=\unitlength,page=3]{catenary.pdf}}%
    \put(0.82017664,0.14710313){\color[rgb]{0,0,0}\makebox(0,0)[lb]{\smash{$p$}}}%
    \put(0.89463179,0.31766898){\color[rgb]{0,0,0}\makebox(0,0)[lb]{\smash{}}}%
    \put(0.56395541,0.21919821){\color[rgb]{0,0,0}\makebox(0,0)[lb]{\smash{$\Gamma_1$}}}%
    \put(0.96890754,0.22372695){\color[rgb]{0,0,0}\makebox(0,0)[lb]{\smash{$\Gamma_2$}}}%
    \put(0.66368797,0.13778347){\color[rgb]{0,0,0}\makebox(0,0)[lb]{\smash{$M$}}}%
    \put(0.31902648,0.51451501){\color[rgb]{0,0,0}\makebox(0,0)[lb]{\smash{}}}%
    \put(0,0){\includegraphics[width=\unitlength,page=4]{catenary.pdf}}%
  \end{picture}%
\endgroup%

%% file: winglike.pdf_tex
\begingroup%
  \makeatletter%
  \providecommand\color[2][]{%
    \errmessage{(Inkscape) Color is used for the text in Inkscape, but the package 'color.sty' is not loaded}%
    \renewcommand\color[2][]{}%
  }%
  \providecommand\transparent[1]{%
    \errmessage{(Inkscape) Transparency is used (non-zero) for the text in Inkscape, but the package 'transparent.sty' is not loaded}%
    \renewcommand\transparent[1]{}%
  }%
  \providecommand\rotatebox[2]{#2}%
  \ifx\svgwidth\undefined%
    \setlength{\unitlength}{595.27559055bp}%
    \ifx\svgscale\undefined%
      \relax%
    \else%
      \setlength{\unitlength}{\unitlength * \real{\svgscale}}%
    \fi%
  \else%
    \setlength{\unitlength}{\svgwidth}%
  \fi%
  \global\let\svgwidth\undefined%
  \global\let\svgscale\undefined%
  \makeatother%
  \begin{picture}(1,0.54761905)%
    \put(0.89463185,0.31766901){\color[rgb]{0,0,0}\makebox(0,0)[lb]{\smash{}}}%
    \put(0.73712453,0.4201227){\color[rgb]{0,0,0}\makebox(0,0)[lb]{\smash{}}}%
    \put(0.1050307,1.44338081){\color[rgb]{0,0,0}\makebox(0,0)[lt]{\begin{minipage}{0.33800789\unitlength}\raggedright \end{minipage}}}%
    \put(0.73297393,0.41529301){\color[rgb]{0,0,0}\makebox(0,0)[lb]{\smash{}}}%
    \put(0.73616966,0.41725827){\color[rgb]{0,0,0}\makebox(0,0)[lb]{\smash{}}}%
    \put(0.38288464,0.37906528){\color[rgb]{0,0,0}\makebox(0,0)[lb]{\smash{}}}%
    \put(0,0){\includegraphics[width=\unitlength,page=1]{winglike.pdf}}%
    \put(0.72092371,0.38273266){\color[rgb]{0,0,0}\makebox(0,0)[lb]{\smash{$\gamma_\lambda$}}}%
    \put(0,0){\includegraphics[width=\unitlength,page=2]{winglike.pdf}}%
    \put(0.70963228,0.08096388){\color[rgb]{0,0,0}\makebox(0,0)[lb]{\smash{$s=0$}}}%
    \put(0,0){\includegraphics[width=\unitlength,page=3]{winglike.pdf}}%
    \put(0.51742302,0.38264904){\color[rgb]{0,0,0}\makebox(0,0)[lb]{\smash{$s=s_0$}}}%
    \put(0.49650874,1.31734398){\color[rgb]{0,0,0}\makebox(0,0)[lt]{\begin{minipage}{0.17759736\unitlength}\raggedright \end{minipage}}}%
    \put(0,0){\includegraphics[width=\unitlength,page=4]{winglike.pdf}}%
    \put(0.6802244,0.51019764){\color[rgb]{0,0,0}\makebox(0,0)[lb]{\smash{$s=s_1$}}}%
    \put(0,0){\includegraphics[width=\unitlength,page=5]{winglike.pdf}}%
    \put(0.61566121,0.47728551){\color[rgb]{0,0,0}\makebox(0,0)[lb]{\smash{$R$}}}%
    \put(0,0){\includegraphics[width=\unitlength,page=6]{winglike.pdf}}%
  \end{picture}%
\endgroup%

%% file: estimate.pdf_tex
\begingroup%
  \makeatletter%
  \providecommand\color[2][]{%
    \errmessage{(Inkscape) Color is used for the text in Inkscape, but the package 'color.sty' is not loaded}%
    \renewcommand\color[2][]{}%
  }%
  \providecommand\transparent[1]{%
    \errmessage{(Inkscape) Transparency is used (non-zero) for the text in Inkscape, but the package 'transparent.sty' is not loaded}%
    \renewcommand\transparent[1]{}%
  }%
  \providecommand\rotatebox[2]{#2}%
  \ifx\svgwidth\undefined%
    \setlength{\unitlength}{595.27559055bp}%
    \ifx\svgscale\undefined%
      \relax%
    \else%
      \setlength{\unitlength}{\unitlength * \real{\svgscale}}%
    \fi%
  \else%
    \setlength{\unitlength}{\svgwidth}%
  \fi%
  \global\let\svgwidth\undefined%
  \global\let\svgscale\undefined%
  \makeatother%
  \begin{picture}(1,0.54761905)%
    \put(0.89463185,0.317669){\color[rgb]{0,0,0}\makebox(0,0)[lb]{\smash{}}}%
    \put(0,0){\includegraphics[width=\unitlength,page=1]{estimate.pdf}}%
    \put(0.34182717,0.37620077){\color[rgb]{0,0,0}\makebox(0,0)[lb]{\smash{$S_\lambda$}}}%
    \put(0.73712453,0.42012271){\color[rgb]{0,0,0}\makebox(0,0)[lb]{\smash{}}}%
    \put(0.1050307,1.44338081){\color[rgb]{0,0,0}\makebox(0,0)[lt]{\begin{minipage}{0.33800789\unitlength}\raggedright \end{minipage}}}%
    \put(0,0){\includegraphics[width=\unitlength,page=2]{estimate.pdf}}%
    \put(0.43219384,0.12115056){\color[rgb]{0,0,0}\makebox(0,0)[lb]{\smash{$P_1$}}}%
    \put(0.73297393,0.41529301){\color[rgb]{0,0,0}\makebox(0,0)[lb]{\smash{}}}%
    \put(0.73616966,0.41725827){\color[rgb]{0,0,0}\makebox(0,0)[lb]{\smash{}}}%
    \put(0.38288464,0.37906528){\color[rgb]{0,0,0}\makebox(0,0)[lb]{\smash{}}}%
    \put(0,0){\includegraphics[width=\unitlength,page=3]{estimate.pdf}}%
    \put(0.9013543,0.12412712){\color[rgb]{0,0,0}\makebox(0,0)[lb]{\smash{$P_1$}}}%
    \put(0.87843851,0.39338761){\color[rgb]{0,0,0}\makebox(0,0)[lb]{\smash{$P_2$}}}%
    \put(0.73712453,0.33836107){\color[rgb]{0,0,0}\makebox(0,0)[lb]{\smash{$M$}}}%
    \put(0.73556817,0.08862196){\color[rgb]{0,0,0}\makebox(0,0)[lb]{\smash{$\Gamma_1$}}}%
    \put(0,0){\includegraphics[width=\unitlength,page=4]{estimate.pdf}}%
    \put(0.73712453,0.42489679){\color[rgb]{0,0,0}\makebox(0,0)[lb]{\smash{$\Gamma_2^t$}}}%
    \put(0,0){\includegraphics[width=\unitlength,page=5]{estimate.pdf}}%
    \put(0.77149822,0.49173452){\color[rgb]{0,0,0}\makebox(0,0)[lb]{\smash{$D\times R$}}}%
    \put(0,0){\includegraphics[width=\unitlength,page=6]{estimate.pdf}}%
  \end{picture}%
\endgroup%